\documentclass[a4paper,12pt]{article}

\usepackage{amsmath}
\usepackage{amsthm}
\usepackage{amssymb}

\usepackage{latexsym}
\usepackage{array}
\usepackage{enumerate, graphicx}

\numberwithin{equation}{section}

\newcommand{\CC}{\mathbb{C}}
\newcommand{\KK}{{\rm K}}

\newcommand{\RR}{\mathbb{R}}
\newcommand{\QQ}{\mathbb{Q}}
\newcommand{\PP}{{\mathbb P}}
\newcommand{\LL}{{\mathbb L}}

\newcommand{\ZZ}{\mathbb{Z}}

\newcommand{\F}{\mathcal{F}}

\newcommand{\Hh}{\mathcal{H}}

\newcommand{\M}{\mathcal{M}}

\newcommand{\sho}{\mathcal{O}}

\newcommand{\HS}{{\rm HS}}
\newcommand{\HSm}{{\rm HS}^{\rm mon}}

\renewcommand{\d}{{\rm dim}}

\newcommand{\Vol}{{\rm Vol}}

\newcommand{\e}{\varepsilon}
\renewcommand{\SS}{{\mathcal S}}
\newcommand{\Spec}{{\rm Spec}}

\newcommand{\id}{{\rm id}}
\newcommand{\supp}{{\rm supp}}

\newcommand{\Int}{{\rm Int}}
\newcommand{\relint}{{\rm rel.int}}

\newcommand{\height}{{\rm ht}}
\newcommand{\1}{{\bf 1}}

\newcommand{\gen}{{\rm gen}}
\newcommand{\Con}{{\rm Con}}
\newcommand{\Var}{{\rm Var}}

\renewcommand{\sp}{{\rm sp}}
\newcommand{\Cone}{{\rm Cone}}

\newcommand{\Db}{{\bf D}^{b}}
\newcommand{\Dbc}{{\bf D}_{c}^{b}}

\newcommand{\tl}[1]{\widetilde{#1}}

\newcommand{\simto}{\overset{\sim}{\longrightarrow}}

\newcommand{\dsum}{\displaystyle \sum}

\newtheorem{definition}{Definition}[section]
\newtheorem{theorem}[definition]{Theorem}
\newtheorem{proposition}[definition]{Proposition}
\newtheorem{lemma}[definition]{Lemma}
\newtheorem{corollary}[definition]{Corollary}

\newtheorem{remark}[definition]{Remark}

\title{Motivic Milnor fibers over complete intersection
varieties and their virtual Betti numbers
}

\author{Alexander ESTEROV
\footnote{Department of Algebra, Faculty of
Mathematics, Complutense University, Plaza
de las Ciencias, 3, Madrid, 28040, Spain.
E-mail: esterov@mccme.ru
(this author is partially supported by the grant
RFBR-10-01-00678) }
and Kiyoshi TAKEUCHI
\footnote{Institute of Mathematics, University  of
Tsukuba, 1-1-1, Tennodai,
Tsukuba, Ibaraki, 305-8571, Japan.
E-mail: takemicro@nifty.com } }

\date{}

\begin{document}

\maketitle

\begin{abstract}
We study the Jordan normal
forms of the local and global
monodromies over complete intersection
subvarieties of $\CC^n$ by using the
theory of motivic Milnor fibers. The
results will be explicitly described by
the mixed volumes of the
faces of Newton polyhedrons.
\end{abstract}

\section{Introduction}\label{sec:1}

In this paper, we study the Jordan normal
forms of the local and global
monodromies over complete intersection
subvarieties of $\CC^n$ with the help of the
theory of motivic Milnor fibers
and their Hodge realizations developed by
Denef-Loeser \cite{D-L-1}, \cite{D-L-2},
Guibert-Loeser-Merle \cite{G-L-M}
and \cite{M-T-4} etc. For $2 \leq k \leq n$ let
\begin{equation}
W= \{ f_1= \cdots =f_{k-1}=0 \} \supset
V= \{ f_1= \cdots =f_{k-1}=f_k=0 \}
\end{equation}
be complete intersection subvarieties
of $\CC^n$ such that $0 \in V$. Assume
that $W$ and $V$ have
isolated singularities at the origin
$0 \in \CC^n$. Then by a fundamental
result of Hamm \cite{Hamm} the Milnor
fiber $F_0$ of $g:=f_k|_W \colon W
\longrightarrow \CC$ at the origin $0$
has the homotopy type of the bouquet
of $(n-k)$-spheres. In particular,
we have $H^j(F_0;\CC) \simeq 0$
($j\neq 0, \ n-k$) and the monodromy
operator $\Phi_{n-k,0}\colon H^{n-k} (F_0;\CC)
\simeq  H^{n-k} (F_0;\CC)$ is called the
$k$-th principal monodromy of
$f:=(f_1, f_2, \ldots , f_k)$.
Under some weak additional
conditions, a beautiful
formula for the eigenvalues of the
$k$-th principal monodromy
$\Phi_{n-k,0}$ was obtained by Oka \cite{Oka-3},
\cite{Oka-2} and Kirillov \cite{Kirillov} etc.
For related results,
see also \cite{E2}, \cite{M-T-1}
and \cite{M-T-2} etc. Moreover the mixed Hodge
structures of the Milnor fiber $F_0$ were
precisely studied by Ebeling-Steenbrink
\cite{E-S} and Tanabe \cite{Tanabe} etc.
However, to the best of our knowledge,
almost nothing is known for
the Jordan normal form of
$\Phi_{n-k,0}$. For a special but important
case, see Dimca \cite{Dim}. In this paper,
we propose a method to describe the Jordan
normal form of $\Phi_{n-k,0}$
in terms of the Newton
polyhedrons of $f_1, f_2, \ldots , f_k$.
Let $\M_{\CC}^{\hat{\mu}}$ be the (localized)
Grothendieck ring of varieties over $\CC$
with (good) action introduced by
Denef-Loeser \cite{D-L-2}.
Then, just as in
Denef-Loeser \cite{D-L-1}, \cite{D-L-2},
and Guibert-Loeser-Merle \cite{G-L-M}
we can introduce an element $\SS_{g,0}$ of
$\M_{\CC}^{\hat{\mu}}$ whose mixed
Hodge numbers carry the information
of $\Phi_{n-k,0}$. We call $\SS_{g,0}$ the
motivic Milnor fiber of $g \colon W
\longrightarrow \CC$ at the origin $0 \in \CC^n$.
We will show that their mixed Hodge numbers can be
easily calculated and hence give
an algorithm to compute the
Jordan normal form of $\Phi_{n-k,0}$.
In order to describe our
results more explicitly, assume also that $f_1,
f_2, \ldots , f_k$ are convenient and
the complete intersection varieties
$W$ and $V$ are non-degenerate at the origin
$0 \in \CC^n$ (see Definition \ref{NDCI}
and \cite{Oka-2} etc.). Set
\begin{equation}
\Gamma_+(f):=\Gamma_+(f_1)+\Gamma_+(f_2)+
\cdots + \Gamma_+(f_k)
\end{equation}
and let $\Gamma_f$ be the union of compact faces of
$\Gamma_+(f)$,
where $\Gamma_+(f_j) \subset \RR_+^n$ is
the Newton polyhedron of $f_j$ at the origin
$0 \in \CC^n$. Then for
each face $\Theta \prec
\Gamma_+(f)$ such that $\Theta \subset
\Gamma_f$ we can naturally define faces
$\gamma_j^{\Theta}$ of $\Gamma_+(f_j)$
($1 \leq j \leq k$) such that
\begin{equation}
\Theta = \gamma_1^{\Theta} + \gamma_2^{\Theta}
+ \cdots + \gamma_k^{\Theta}.
\end{equation}

\begin{center}
\noindent\includegraphics[width=14cm]{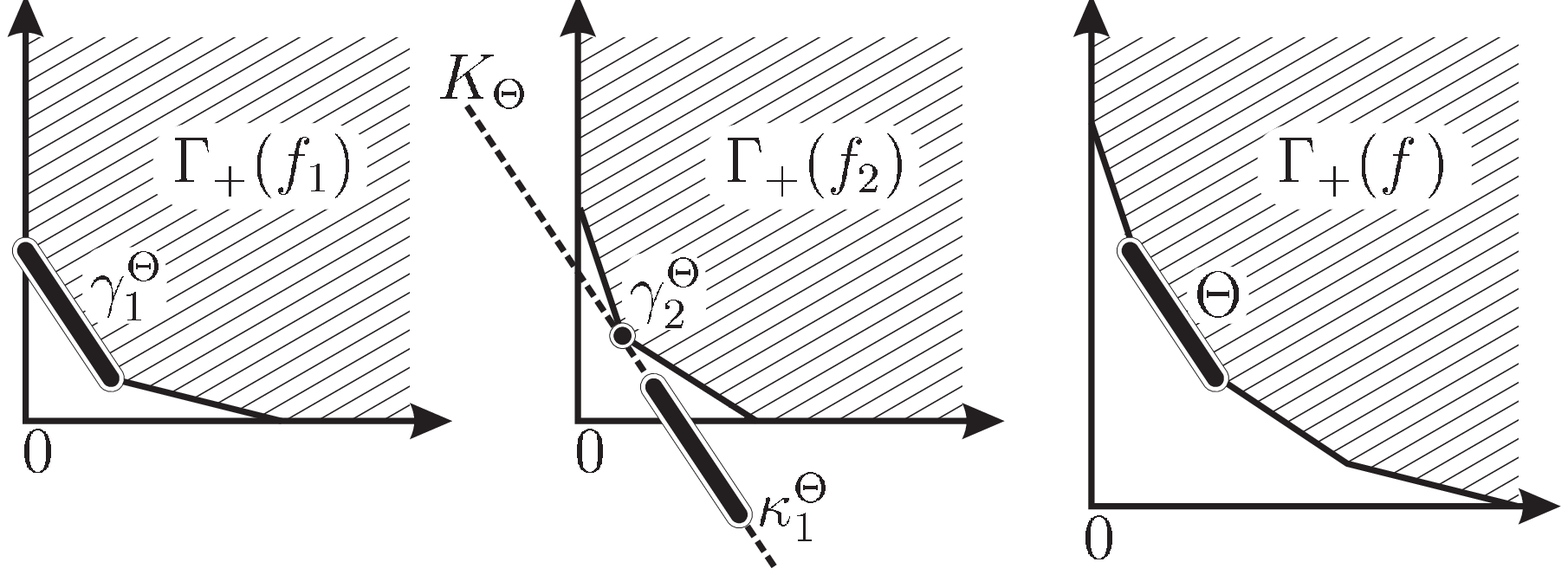} \end{center}

\noindent (The case where $k=2$.
For $K_{\Theta}$ and
$\kappa_1^{\Theta}$
in the figure, see Section \ref{sec:4}).
By using $\gamma_j^{\Theta}$,
we define a non-degenerate complete intersection
subvariety $Z^*_{\Delta_{\Theta}} \subset
(\CC^*)^{\d \Theta +1}$ and
an element $[Z^*_{\Delta_{\Theta}}]
\in \M_{\CC}^{\hat{\mu}}$ (see Section \ref{sec:4}
for the details). Moreover
for such $\Theta \prec \Gamma_+(f)$
let $s_{\Theta}$ be the dimension of the
minimal coordinate subspace of $\RR^n$
containing $\Theta$ and set $m_{\Theta}=s_{\Theta}-
\d \Theta -1\geq 0$. Recall that
the element $[\LL ] \in
\M_{\CC}^{\hat{\mu}}$ called the Lefschetz motive
is defined to be the affine line $\LL \simeq
\CC$ with the trivial action.
Following the notations in
Denef-Loeser \cite[Section 3.1.2 and 3.1.3]{D-L-2},
we denote by $\HSm$ the abelian
category of Hodge structures with a
quasi-unipotent endomorphism.
Its Grothendieck group, whose elements are
formal differences of those of
$\HSm$, is denoted by $\KK_0(\HSm)$.
Let $\chi_h: \M_{\CC}^{\hat{\mu}}
\longrightarrow \KK_0(\HSm)$ be the Hodge
characteristic map in \cite{D-L-2}
(see also Section \ref{sec:4}).
For an element $[H] \in \KK_0(\HSm)$ and
$\lambda \in \CC^*$ we denote by $[H]_{\lambda}
\in \KK_0(\HS )$ the $\lambda$-eigenspace
of the quasi-unipotent endomorphism on $[H]$.

\begin{theorem}\label{GTM}
Assume that $\lambda \in \CC^*\setminus \{1\}$. Then
\begin{enumerate}
\item In the
Grothendieck group $\KK_0(\HS )$, we have
\begin{equation}
\chi_h(\SS_{g,0})_{\lambda}= \dsum_{
\Theta \subset \Gamma_f,
 \d \Theta \geq k-1 }
\chi_h\big((1-\LL)^{m_{\Theta}}
\cdot[Z_{\Delta_{\Theta}}^*]\big)_{\lambda}.
\end{equation}
\item For $i \geq 1$, the number
of the Jordan blocks for the eigenvalue
$\lambda$ with sizes $\geq i$ in
$\Phi_{n-k,0}\colon H^{n-k}(F_0;\CC) \simeq
H^{n-k}(F_0;\CC)$ is equal to
\begin{equation}
(-1)^{n-k} \left\{ \beta_{n-k-1+i}(\SS_{g,0})_{\lambda}
+ \beta_{n-k+i}(\SS_{g,0})_{\lambda} \right\},
\end{equation}
where $\beta_{j}(\SS_{g,0})_{\lambda}$ is the
$j$-th virtual Betti number of $\SS_{g,0}$
(with respect to the eigenvalue $\lambda$)
defined by the weight $j$-part of the
Hodge structure $\chi_h(\SS_{g,0})$
(see Section \ref{sec:3}).
\end{enumerate}
\end{theorem}
By this theorem, for the determination of
the Jordan normal form of $\Phi_{n-k,0}$
concerning the eigenvalues $\lambda \not= 1$,
it suffices to calculate the weight
multiplicities of
$\chi_h([Z^*_{\Delta_{\Theta}}])
\in \KK_0(\HSm)$. In
 Section \ref{sec:4},
by using the Cayley trick in
\cite[Section 6]{D-K} we reduce these
calculations to those for non-degenerate
hypersurfaces in algebraic tori with
action.
Since we can always calculate the mixed Hodge
numbers of such hypersurfaces by our
previous results in \cite[Section 2]{M-T-4},
we thus obtain an algorithm to compute the
Jordan normal form of $\Phi_{n-k,0}$. Moreover,
in Section \ref{sec:5} we give some
closed formulas for the numbers of
the Jordan blocks for the eigenvalues
$\lambda \not= 1$ with large sizes in
$\Phi_{n-k,0}$. The results will be
explicitly described by the mixed
volumes of the faces $\gamma_j^{\Theta}$
of $\Gamma_+(f_j)$ ($1 \leq j \leq k$).
See Theorems \ref{MPS} and
\ref{SMB} for the details.
In the course of the proof of these results,
Proposition \ref{AE}, which
generalizes a result of Khovanskii
\cite{Khovanskii}, will play a central
role. This proposition expresses an alternating
sum of the numbers of certain lattice points
defined by a polytope by its volume. We believe
that it will be very useful in the study of
lattice points in non-integral polytopes.
Indeed, by Proposition \ref{AE}
we can rewrite the main results
of \cite{M-T-4} much more simply in terms of
the volumes of polytopes. Unfortunately,
for some technical reason, by
our methods we cannot obtain
similar results for the eigenvalue $1$
(see Remark \ref{EV1} below).
In Section \ref{sec:6}, we will show that
our methods are useful also in the study of
the monodromies at infinity
over complete intersection
subvarieties of $\CC^n$. Even in
this global case, we obtain
results completely parallel to
the ones in the local case. We
thus find a beautiful symmetry between
local and global as in \cite{M-T-4}.

\medskip

\noindent{\bf Acknowledgement:} The authors would
like to express their hearty
gratitude to Professors Y. Matsui,
A. Melle-Hern{\'a}ndez and C. Sabbah for their
very fruitful discussions.

\section{Preliminary notions and results}\label{sec:2}

In this section, we introduce basic
notions and results which will be used
in this paper. In this paper, we essentially
follow the terminology of
\cite{Dimca}, \cite{H-T-T} and \cite{K-S}
etc. For example, for a
topological space $X$ we denote by
$\Db(X)$ the derived category whose
objects are bounded complexes of sheaves
of $\CC_X$-modules on $X$. Moreover
if $X$ is an algebraic variety over $\CC$
we denote by $\Dbc(X)$ the full
subcategory of $\Db(X)$ consisting of
constructible complexes of sheaves.
Let $f : X \longrightarrow \CC$
be a non-constant regular function on
an algebraic variety $X$ over $\CC$ and
set $X_0= \{x\in X\ |\ f(x)=0\}
\subset X$.
Then we have
the nearby cycle functor
\begin{equation}
\psi_f : \Db(X) \longrightarrow \Db(X_0)
\end{equation}
which preserves the constructibility
(see \cite{Dimca} and \cite{K-S} etc.).
As we see in Proposition \ref{prp:2-7} below,
the nearby cycle functor $\psi_f$
generalizes the classical notion of
Milnor fibers. First, let us recall the
definition of Milnor fibers
 over singular varieties
(see for example \cite{Takeuchi} for a
review on this subject). Let $X$ be a
subvariety of $\CC^m$ and $f \colon X
\longrightarrow \CC$ a non-constant
regular function on $X$. Namely we
assume that there exists a polynomial
function $\tl{f} \colon \CC^m \longrightarrow
\CC$ on $\CC^m$ such that
$\tl{f}|_X=f$. For simplicity, assume
also that the origin $0 \in \CC^m$ is
contained in $X_0= \{x\in X\ |\ f(x)=0\}$.
Then the following lemma is
well-known.

\begin{lemma}{\rm \bf(
\cite[Theorem 1.1]{Le})}\label{lem:2-5}
For sufficiently small $\e >0$, there
exists $\eta_0 >0$ with $0<\eta_0 <<
\e$ such that for $0 < \forall \eta <\eta_0$
the restriction of $f$:
\begin{equation}
X \cap B(0;\e) \cap \tl{f}^{-1}(D(0;\eta)
\setminus \{0\}) \longrightarrow
D(0;\eta) \setminus \{0\}
\end{equation}
is a topological fiber bundle over
the punctured disk $D(0;\eta) \setminus
\{0\}:=\{ z \in \CC \ |\ 0<|z|<\eta \}$,
where $B(0;\e)$ is the open ball in
$\CC^m$ with radius $\e$ centered at the origin.
\end{lemma}

\begin{definition}\label{dfn:2-6}
A fiber of the above fibration is
called the Milnor fiber of the function $f:
 X\longrightarrow \CC$ at $0 \in X_0$
and we denote it by $F_0$.
\end{definition}

\begin{proposition}{\bf(
\cite[Proposition 4.2.2]{Dimca})}\label{prp:2-7}
There exists a natural isomorphism
\begin{equation}
H^j(F_0;\CC) \simeq H^j(\psi_f(\CC_X))_0
\end{equation}
for any $j \in \ZZ$.
\end{proposition}
Recall also that in the above
situation, as in the
case of polynomial functions over
$\CC^n$ (see Milnor \cite{Milnor}), we can
define the Milnor monodromy operators
\begin{equation}
\Phi_{j,0} \colon H^j(F_0;\CC)
\overset{\sim}{\longrightarrow}
H^j(F_0;\CC) \ \qquad \ (j=0,1,\ldots).
\end{equation}
Similarly, also for any $y \in
X_0=\{x \in X \ |\ f(x)=0\}$ we can define
the Milnor fiber $F_y$ and
its monodromies $\Phi_{j,y}$.
The notion of Milnor monodromies can be
also generalized as follows.
Let $\F \in \Dbc(X)$. Then there
exists a monodromy automorphism
\begin{equation}
\Phi(\F) \colon \psi_f(\F) \simto \psi_f(\F)
\end{equation}
of $\psi_f(\F)$ in $\Dbc(X_0)$
(see \cite{Dimca} and \cite{K-S} etc.).

\medskip

Next we recall
Bernstein-Khovanskii-Kushnirenko's
theorem \cite{Khovanskii}.

\begin{definition}
Let $g(x)=\sum_{v \in \ZZ^n}
a_vx^v$ be a Laurent polynomial on $(\CC^*)^n$
($a_v\in \CC$).
\begin{enumerate}
\item We call the convex hull
of $\supp(g):=\{v\in \ZZ^n \ |\ a_v \neq 0\}
\subset \ZZ^n \subset \RR^n$ in $\RR^n$
the Newton polytope of $g$ and
denote it by $NP(g)$.
\item For a vector $u\in \RR^n$, we set
\begin{equation}
\Gamma(g;u):=\left\{ v\in NP(g) \ \left|
\ \langle u,v \rangle =\min_{w \in
NP(g)} \langle u,w \rangle \right. \right\},
\end{equation}
where for $u=(u_1,\ldots,u_n)$ and $v=(v_1,\ldots, v_n)$
we set $\langle
u,v\rangle =\sum_{i=1}^n u_iv_i$.
We call $\Gamma(g;u)$ the supporting face
of $u$ in $NP(g)$.
\item For a vector $u \in \RR^n$,
we define the $u$-part of $g$ by
\begin{equation}
g^u(x):=\dsum_{v \in \Gamma(g;u) \cap
\ZZ^n} a_vx^v.
\end{equation}
\end{enumerate}
\end{definition}

\begin{definition}
Let $g_1, g_2, \ldots , g_p$ be
Laurent polynomials on $(\CC^*)^n$. Then we
say that the subvariety $Z^*=\{ x\in (\CC^*)^n
\ | \ g_1(x)=g_2(x)= \cdots
=g_p(x)=0 \}$ of $(\CC^*)^n$ is
non-degenerate complete intersection if for
any vector $u \in \RR^n$ the $p$-form
$dg_1^u \wedge dg_2^u \wedge \cdots
\wedge dg_p^u$ does not vanish on
$\{ x\in (\CC^*)^n \ |\ g_1^u(x)= \cdots
=g_p^u(x)=0 \}$.
\end{definition}

\begin{theorem}[Bernstein-Khovanskii-Kushnirenko's
theorem, see \cite{Khovanskii} etc.]\label{BKK}
Let $g_1, g_2, \ldots , g_p$ be
Laurent polynomials on $(\CC^*)^n$. Assume
that the subvariety $Z^*=\{ x\in (\CC^*)^n
\ |\ g_1(x)=g_2(x)= \cdots
=g_p(x)=0 \}$ of $(\CC^*)^n$ is
non-degenerate complete intersection. Set
$\Delta_i:=NP(g_i)$ for $i=1,\ldots, p$.
Then the Euler characteristic $\chi (Z^*)$
of $Z^*$ is given by
\begin{equation}
\chi(Z^*)=(-1)^{n-p}\dsum_{\begin{subarray}{c}
a_1,\ldots,a_p \geq 1\\
a_1+\cdots
+a_p=n
\end{subarray}}
MV(
\underbrace{\Delta_1,\ldots,\Delta_1}_{\text{$a_1$-times}},
\ldots,
\underbrace{\Delta_p,\ldots,\Delta_p}_{\text{$a_p$-times}}),
\end{equation}
where
$MV(
\underbrace{\Delta_1,\ldots,\Delta_1}_{\text{$a_1$-times}},
\ldots,
\underbrace{\Delta_p,\ldots,\Delta_p}_{\text{$a_p$-times}})
\in \ZZ$ is the normalized $n$-dimensional
mixed volume
with respect to the lattice $\ZZ^n \subset \RR^n$.
\end{theorem}

\begin{remark}\label{rem:2-13}
Let $Q_1,Q_2,\ldots, Q_n$ be
polytopes in $\RR^n$. Then
their normalized $n$-dimensional mixed volume
$MV(Q_1,Q_2,\ldots,Q_n) \in \RR_+$ is
given by the formula
\begin{equation}
MV(Q_1, Q_2, \ldots , Q_n)
=\frac{1}{n!}\dsum_{k=1}^n (-1)^{n-k}
\left\{ \sum_{\begin{subarray}{c}I\subset
\{1,\ldots,n\}\\ \sharp
I=k\end{subarray}}\Vol_{\ZZ}
\left(\dsum_{i\in I}Q_i\right)\right\},
\end{equation}
where $\Vol_{\ZZ}(*)\in
\ZZ$ is the normalized $n$-dimensional
volume (i.e. the $n!$ times the usual volume).
Note that if $Q_1, \ldots, Q_n$ are integral polytopes
$MV(Q_1,Q_2,\ldots,Q_n) \in \ZZ_+$.
\end{remark}

Finally we shall introduce our
recent results in \cite[Section 2]{M-T-4}.
From now on, let us fix an element
$\tau =(\tau_1,\ldots, \tau_n) \in
T:=(\CC^*)^n$ and let $g$ be a Laurent
polynomial on $(\CC^*)^n$ such that
$Z^*=\{ x\in (\CC^*)^n \ |\ g(x)=0\}$
is non-degenerate and invariant by the
automorphism $l_{\tau} \colon (\CC^*)^n
\underset{\tau
\times}{\simto}(\CC^*)^n$ induced
by the multiplication by $\tau$. Set
$\Delta =NP(g)$ and for simplicity
assume that $\d \Delta=n$. Then there
exists $\beta \in \CC$ such that
$l_{\tau}^*g= g \circ l_{\tau}=\beta g$.
This implies that for any vertex $v$
of $\Delta =NP(g)$ we have
${\tau}^v={\tau}_1^{v_1} \cdots
{\tau}_n^{v_n}=\beta$. Moreover by the
condition $\d \Delta=n$ we see
that $\tau_1, \tau_2, \ldots , \tau_n$ are
roots of unity. For $p,q \geq 0$
and $k \geq 0$, let
$h^{p,q}(H_c^k(Z^*;\CC))$ be the
mixed Hodge number of $H_c^k(Z^*;\CC)$ and set
\begin{equation}
e^{p,q}(Z^*)=\dsum_k (-1)^k
h^{p,q}(H_c^k(Z^*;\CC))
\end{equation}
as in \cite{D-K}. The above automorphism
of $(\CC^*)^n$ induces a morphism
of mixed Hodge structures $l_{\tau}^* :
 H_c^k(Z^*;\CC) \simto
H_c^k(Z^*;\CC)$ and hence $\CC$-linear
transformations on the $(p,q)$-parts
$H_c^k(Z^*;\CC)^{p,q}$ of $H_c^k(Z^*;\CC)$.
For $\alpha \in \CC$, let
$h^{p,q}(H_c^k(Z^*;\CC))_{\alpha}$
be the dimension of the
$\alpha$-eigenspace $H_c^k(Z^*;\CC)_{\alpha}^{p,q}$
of this automorphism of
$H_c^k(Z^*;\CC)^{p,q}$ and set
\begin{equation}
e^{p,q}(Z^*)_{\alpha}=
\dsum_k (-1)^k h^{p,q}(H_c^k(Z^*;\CC))_{\alpha}.
\end{equation}
Since we have $l_{\tau}^r =\id_{Z^*}$
for $r >> 0$, these numbers are zero
unless $\alpha$ is a root of unity.
Moreover we have
\begin{equation}
e^{p,q}(Z^*)=\dsum_{\alpha \in \CC}
e^{p,q}(Z^*)_{\alpha}, \qquad
e^{p,q}(Z^*)_{\alpha}=
e^{q,p}(Z^*)_{\overline{\alpha}}.
\end{equation}
In this situation, along the lines
of Danilov-Khovanskii \cite{D-K} we can
give an algorithm for computing these
numbers $e^{p,q}(Z^*)_{\alpha}$ as
follows. First of all, as in
\cite[Section 3]{D-K} we obtain the following
Lefschetz type theorem.

\begin{proposition}{\bf
(\cite[Proposition 2.6]{M-T-4})}\label{prp:2-15}
For $p,q \geq 0$ such that $p+q >n-1$, we have
\begin{equation}
e^{p,q}(Z^*)_{\alpha}=
\begin{cases}
(-1)^{n+p+1}\binom{n}{p+1} & (
\text{$\alpha=1$ and $p=q$}),\\
 \ \qquad \ 0 & (\text{otherwise}).
\end{cases}
\end{equation}
\end{proposition}
For a vertex $w$ of $\Delta$, consider
the translated polytope
$\Delta^w:=\Delta -w$ such that $0 \prec
\Delta^w$ and ${\tau}^v=1$ for any
vertex $v$ of $\Delta^w$. Then for
$\alpha \in \CC$ and $k \geq 0$ set
\begin{equation}
l^*(k\Delta)_{\alpha}=\sharp \{ v \in
\Int (k\Delta^w) \cap \ZZ^n \ |\
{\tau}^v =\alpha\} \in \ZZ_+:=\ZZ_{\geq 0}
\end{equation}
We can easily see that these numbers
$l^*(k\Delta)_{\alpha}$ do not depend on
the choice of the vertex $w$ of
$\Delta$. Next, define a formal power
series $P_{\alpha}(\Delta;t)=\sum_{i
\geq 0} \varphi_{\alpha, i}(\Delta)t^i$ by
\begin{equation}
P_{\alpha}(\Delta;t)=(1-t)^{n+1}
\left\{ \dsum_{k \geq 0}
l^*(k\Delta)_{\alpha}t^k\right\}.
\end{equation}
Then we can easily show
that $P_{\alpha}(\Delta;t)$ is
actually a polynomial as in
\cite[Section 4.4]{D-K}.

\begin{theorem}{\bf (
\cite[Theorem 2.7]{M-T-4})}\label{thm:2-14}
In the situation as above, we have
\begin{equation}
\dsum_q e^{p,q}(Z^*)_{\alpha}
=\begin{cases}
(-1)^{p+n+1}\binom{n}{p+1} +(-1)^{n+1}
\varphi_{\alpha, n-p}(\Delta) &
(\alpha=1), \\
(-1)^{n+1} \varphi_{\alpha, n-p}
(\Delta) & (\alpha \neq 1)
\end{cases}
\end{equation}
(we used the convention $\binom{a}{b}=0$
($0 \leq a <b$) for binomial
coefficients).
\end{theorem}
By Proposition \ref{prp:2-15} and Theorem
\ref{thm:2-14}, we obtain an
algorithm to calculate the numbers
$e^{p,q}(Z^*)_{\alpha}$ of the
non-degenerate hypersurface $Z^* \subset
(\CC^*)^n$ for any $\alpha \in \CC$
as in \cite[Section 5.2]{D-K}. Indeed
for a projective toric
compactification $X$ of $(\CC^*)^n$ such
that the closure $\overline{Z^*}$
of $Z^*$ in $X$ is smooth, the variety
$\overline{Z^*}$ is smooth projective
and hence there exists a perfect pairing
\begin{equation}
H^{p,q}(\overline{Z^*};\CC)_{\alpha}
\times H^{n-1-p,
n-1-q}(\overline{Z^*};\CC)_{\alpha^{-1}}
\longrightarrow \CC
\end{equation}
for any $p,q \geq 0$ and $\alpha \in \CC^*$
(see for example \cite[Section 5.3.2]{Voisin}).
Therefore, we obtain equalities
$e^{p,q}(\overline{Z^*})_{\alpha}=
e^{n-1-p,n-1-q}(\overline{Z^*})_{\alpha^{-1}}$
which are necessary to proceed the
algorithm in \cite[Section 5.2]{D-K}.
The following notion is very useful to
construct such compactifications
of $(\CC^*)^n$.

\begin{definition}\label{MJR}(\cite{D-K})
\begin{enumerate}
\item
Let $\Delta$ be an $n$-dimensional integral
polytope in $(\RR^n, \ZZ^n)$.
For a vertex $w$ of $\Delta$, we define
a closed convex cone $\Con (\Delta,
w)$ by $\Con (\Delta,w)=\{ r \cdot (v -w)
\ |\ r \in \RR_+, \ v \in \Delta \}
\subset \RR^n$.
\item
Let $\Delta$ and $\Delta^{\prime}$
be two $n$-dimensional integral polytopes
in $(\RR^n, \ZZ^n)$. We denote by
${\rm som}(\Delta )$ (resp.
${\rm som}(\Delta^{\prime})$) the set of
vertices of $\Delta$ (resp. $\Delta^{\prime}$).
Then we say that
$\Delta^{\prime}$ majorizes $\Delta$ if
there exists a map $\Psi \colon {\rm
som}(\Delta^{\prime}) \longrightarrow
{\rm som}(\Delta)$ such that
$\Con(\Delta, \Psi(w)) \subset
\Con(\Delta^{\prime}, w)$ for any
$w \in {\rm som}(\Delta^{\prime})$.
\end{enumerate}
\end{definition}
Note that if $\Delta^{\prime}$ majorizes $\Delta$
the map $\Psi \colon {\rm
som}(\Delta^{\prime}) \longrightarrow
{\rm som}(\Delta)$ is unique (see \cite{D-K}).
For an $n$-dimensional integral polytope $\Delta$
in $(\RR^n, \ZZ^n)$, we denote by
$X_{\Delta}$ the (projective) toric variety
associated with the dual fan of $\Delta$.
Recall also that if $\Delta^{\prime}$
majorizes $\Delta$ then there exists a natural
morphism $X_{\Delta^{\prime}}
\longrightarrow X_{\Delta}$. Now we return
to the original situation.
For $\alpha \in \CC$ we define the $\alpha$-Euler
characteristic $\chi(Z^*)_{\alpha} \in \ZZ$
of $Z^* \subset (\CC^*)^n$ by
\begin{equation}
\chi(Z^*)_{\alpha}=
\dsum_{p,q} e^{p,q}(Z^*)_{\alpha}.
\end{equation}
Then we have the following result.
\begin{proposition}\label{EU}
For any $\alpha \in \CC$ we have
\begin{equation}
\chi(Z^*)_{\alpha}
=\begin{cases}
(-1)^{n-1}+ \sum_{k=1}^n(-1)^{k+1}
\binom{n}{k} l^*(k \Delta )_{\alpha}
 & (\alpha=1), \\
\sum_{k=1}^n(-1)^{k+1}
\binom{n}{k} l^*(k \Delta )_{\alpha}
 & (\alpha \neq 1).
\end{cases}
\end{equation}
\end{proposition}
\begin{proof}
By Theorem \ref{thm:2-14} we have
\begin{equation}
\chi(Z^*)_{\alpha}
=\begin{cases}
(-1)^{n+1} \{ 1+ \varphi_{1, 0}(\Delta) +
\cdots + \varphi_{1, n}(\Delta) \} &
(\alpha=1), \\
(-1)^{n+1} \{ \varphi_{\alpha , 0}(\Delta) +
\cdots + \varphi_{\alpha ,
n}(\Delta)\} & (\alpha \neq 1).
\end{cases}
\end{equation}
Then the result follows from
\begin{eqnarray}
\sum_{j=0}^n \varphi_{\alpha ,j}(\Delta)
& = &
\sum_{j=0}^n \sum_{i=0}^j
(-1)^i \binom{n+1}{i}
l^*((j-i)\Delta)_{\alpha}
\\
& = &
\sum_{k=0}^n (-1)^{n-k} \binom{n}{k}
l^*(k \Delta )_{\alpha}.
\end{eqnarray}
\end{proof}
From now on, assume also that for any vertex
$v$ of $\Delta$ we have $\tau^v=1$. Let $L_{\tau}
\subset \ZZ^n$ be the sublattice of $\ZZ^n$
defined by $L_{\tau}=\{ v \in \ZZ^n \ | \
 \tau^v=1 \}$. For an integral polytope
$\Box$ in $\RR^n$ we set $\natural (\Box)=
(-1)^{\d \Box} \sharp \{ \relint (\Box) \cap
L_{\tau}\}$, where $\relint (\Box)$ is the
relative interior of $\Box$.
Note that if $\d \Box =0$ then
we have $\natural (\Box)=1$ or $0$ depending
on whether $\Box$ is a point in $L_{\tau}$
or not. Then for any $\alpha \in \CC$ such
that $\{ v \in \ZZ^n \ | \
 \tau^v=\alpha \} \not= \emptyset$, by
taking a point $w(\alpha) \in
\{ v \in \ZZ^n \ | \  \tau^v=\alpha \}$,
Proposition \ref{EU} can be rewritten as
follows.
\begin{proposition}\label{EUP}
For any $\alpha \in \CC$ we have
\begin{equation}
(-1)^{n-1} \chi(Z^*)_{\alpha} =
\sum_{k=0}^n (-1)^{k} \binom{n}{k}
\cdot \natural (k \Delta -w(\alpha)).
\end{equation}
\end{proposition}
More generally, for any subvariety $Y^*
\subset (\CC^*)^n$ which is invariant by
$l_{\tau} \colon (\CC^*)^n \simto (\CC^*)^n$,
$p,q \geq 0$ and $\alpha \in \CC$ we can
define $e^{p,q}(Y^*)_{\alpha} \in \ZZ$
satisfying similar properties. For example,
fix integral polytopes $\Delta_1, \ldots, \Delta_k$
in $\RR^n$ whose all vertices lie in $L_{\tau}$
and set $N=\sum_{j=1}^k \sharp (\Delta_j \cap
L_{\tau})$. Let $S \simeq \CC^N$ be the
set of $k$-tuples $(g_1, \ldots, g_k)$ of
Laurent polynomials on $(\CC^*)^n$ such
that $\supp g_j \subset \Delta_j \cap
L_{\tau}$. Here we regard $S$ as
the affine space $\CC^N$
consisting of the coefficients of
$(g_1, \ldots, g_k)$.
Let $S_{\gen}$ be the subset of $S$ consisting of
$k$-tuples $(g_1, \ldots, g_k)$ such that
$NP(g_j)=\Delta_j$ and $Z(g_1, \ldots, g_k)^*:=
\{ x \in (\CC^*)^n \ | \
g_1(x)= \cdots =g_k(x)=0 \}$ is
a non-degenerate C.I. Then $S_{\gen}$
is open dense in
$S \simeq \CC^N$ and
for any $(g_1, \ldots, g_k) \in S_{\gen}$
the C.I. subvariety
$Z(g_1, \ldots, g_k)^* \subset (\CC^*)^n$
is invariant by
$l_{\tau}$. Hence the numbers
$e^{p,q}(Z(g_1, \ldots, g_k)^*
)_{\alpha} \in \ZZ$ are defined.
\begin{lemma}\label{INV}
For any $p,q \geq 0$ and $\alpha \in \CC$ the
integer $e^{p,q}(Z(g_1, \ldots, g_k)^*
)_{\alpha}$ does not depend on
$(g_1, \ldots, g_k) \in S_{\gen}$.
\end{lemma}
\begin{proof}
Set $m= \dim (\Delta_1+ \cdots +\Delta_k)$.
Then for any $(g_1, \ldots, g_k) \in S_{\gen}$
we have $Z(g_1, \ldots, g_k)^* \simeq
(\CC^*)^{n-m} \times Z^{\prime}(g_1, \ldots, g_k)^*$
for a non-degenerate complete intersection
$Z^{\prime}(g_1, \ldots, g_k)^* \subset
(\CC^*)^m$ and $l_{\tau} : Z(g_1, \ldots, g_k)^*
\simto Z(g_1, \ldots, g_k)^*$ is homotopic
to $\id_{(\CC^*)^{n-m}} \times l_{\tau^{\prime}}$
for some $\tau^{\prime} \in (\CC^*)^m$.
So we may assume that
$\dim (\Delta_1+ \cdots +\Delta_k)=n$ from
the first. Let $\Delta$ be an integral polytope
in $\RR^n$ which majorizes $\Delta_1+
\cdots +\Delta_k$. Then by subdividing the
dual fan $\Sigma_1$ of $\Delta$ in $\RR^n$
we obtain a complete fan $\Sigma$ such that the
toric variety $X_{\Sigma}$ associated to it
is a smooth compactification of $(\CC^*)^n$.
By construction, the closure
$\overline{Z(g_1, \ldots, g_k)^*}$ of
$Z(g_1, \ldots, g_k)^*$ in $X_{\Sigma}$
is smooth for any $(g_1, \ldots, g_k)
\in S_{\gen}$, and hence we obtain a
family $\pi : \overline{Z^*}
\longrightarrow S_{\gen}$ of smooth
projective varieties over $S_{\gen}$.
By using the relative de Rham complex
$\Omega^{\cdot}_{\overline{Z^*}/S_{\gen}}$
of $\pi : \overline{Z^*}
\longrightarrow S_{\gen}$, for each $i \geq 0$
we obtain a holomorphic variation
$\Hh^i=R^i \pi_*(
\Omega^{\cdot}_{\overline{Z^*}/S_{\gen}})$
of (pure) Hodge structures on $S_{\gen}$.
Its Hodge filtration $F^p \Hh^i \subset \Hh^i$
($p \geq 0$) is defined by
$F^p \Hh^i=R^i \pi_*(
\Omega^{\geqslant p}_{\overline{Z^*}/S_{\gen}})$
(see \cite[Section 10.2.1]{Voisin} etc.). For
$p \geq 0$ let
\begin{equation}
\Phi (p): F^p \Hh^i \longrightarrow F^p \Hh^i
\end{equation}
be the $\sho_{S_{\gen}}$-linear endomorphism
of the locally free $\sho_{S_{\gen}}$-module
$F^p \Hh^i$ induced by the pull-back of
$\Omega^{\geqslant p}_{\overline{Z^*}/S_{\gen}}$
by $l_{\tau} \times \id_{S_{\gen}}$. Then
by $(\Phi (p))^N= \id$ ($N>>0$) we obtain
a decomposition
\begin{equation}
F^p \Hh^i =\bigoplus_{\alpha \in \CC^*}
(F^p \Hh^i)_{\alpha}
\end{equation}
of $F^p \Hh^i$ into the eigenspaces
of $\Phi (p)$. Hence $(F^p \Hh^i)_{\alpha}$
are also locally free over $\sho_{S_{\gen}}$.
Then by
\cite[Theorem 10.10]{Voisin} the functions
$e^{p,q}(\overline{Z(g_1, \ldots, g_k)^*}
)_{\alpha}$ on $S_{\gen}$ are constant.
Moreover we can easily prove a similar
statement also for $e^{p,q}(Z(g_1,
\ldots, g_k)^* )_{\alpha}$ by
induction on $n$.
\end{proof}

\section{Combinatorial results and their
applications}\label{sec:3}

In this section, we shall describe the $\alpha$-Euler
characteristic $\chi(Z^*)_{\alpha} \in \ZZ$
of the non-degenerate hypersurface
$Z^* \subset (\CC^*)^n$ introduced in Section \ref{sec:2}
in terms of the volume of
its Newton polytope $\Delta$. For this purpose,
we first consider the following more general
situation. Let $L \subset \ZZ^n$ be a sublattice
of rank $n$. For a bounded subset $A$ of $\RR^n$
(resp. a polytope $\Box$ in $\RR^n$)
we set $\sharp A=\sharp (A \cap L) \in \ZZ_+$
(resp. $\natural (\Box)=(-1)^{\d \Box}
\sharp \{ \relint (\Box) \cap
L \} \in \ZZ$) for short. Let $\Delta_1, \Delta_2,
\ldots, \Delta_n$ be integral polytopes in $\RR^n$
whose all vertices lie in $L$. For a subset
$J \subset \{ 1,2, \ldots, n\}$ we set $\Delta_J
=\sum_{j \in J}\Delta_j$. In particular, for
$J = \emptyset$ we set $\Delta_J=\{ 0\}$.
Then the following result is well-known.
\begin{theorem}\label{KH} (Khovanskii
\cite{Khovanskii})
In the situation as above, we have
\begin{equation}
\dsum_{J \subset \{ 1,2, \ldots, n\} }
(-1)^{\sharp J}
\natural (\Delta_J)=MV(\Delta_1, \Delta_2,
\ldots, \Delta_n),
\end{equation}
where $MV(\Delta_1,
\ldots, \Delta_n) \in \ZZ_+$ is the
normalized $n$-dimensional mixed volume
of $\Delta_1, \Delta_2,
\ldots, \Delta_n$ with respect to the
lattice $L$.
\end{theorem}
From now on, we will
generalize this theorem
as follows. Let $\Delta_1,
\ldots, \Delta_n$ be as above and
pick another polytope $\Delta_0$
in $\RR^n$ (which is not assumed to
be integral). Also for a subset
$I \subset \{ 0, 1,2, \ldots, n\}$
we set $\Delta_I =\sum_{j \in I}\Delta_j$.
\begin{proposition}\label{AE}
In the situation as above, we have
\begin{eqnarray}\label{aim}
MV(\Delta_1, \Delta_2, \ldots, \Delta_n)
& = &
\dsum_{J \subset \{ 1,2, \ldots, n\} }
(-1)^{\sharp J}
\natural (\Delta_{\{ 0\} \sqcup J})
\\
& = &
\dsum_{J \subset \{ 1,2, \ldots, n\} }
(-1)^{n- \sharp J}
\sharp (\Delta_{\{ 0\} \sqcup J}).
\end{eqnarray}
\end{proposition}
\begin{proof}
The proof proceeds in three steps.
\par \noindent (A) Assume that there
exists $1 \leq j \leq n$ such that
$\d \Delta_j=0$. In this case, the
mixed volume $MV(\Delta_1,
\ldots, \Delta_n)$ is zero and the
other two terms in \eqref{aim} also vanish,
because for each $J \subset \{ 1,2, \ldots, n\}$
such that $j \notin J$ we have the cancelling
\begin{equation}
(-1)^{\sharp (J \sqcup \{ j\})}
\natural (\Delta_{\{ 0\} \sqcup
J \sqcup \{ j\} })+(-1)^{\sharp J}
\natural (\Delta_{\{ 0\} \sqcup J})=0
\end{equation}
etc.
\par \noindent (B) Assume that
$\Delta_1, \Delta_2, \ldots, \Delta_n$ are
linearly independent segments and
$\Delta_0$ consists of one point $p \in \RR^n$.
In this case, for each $1 \leq j \leq n$ by taking
a vertex $q_j$ of the segment
$\Delta_j$ we set $\hat{\Delta}_j
=\Delta_j \setminus \{ q_j \}$. Then we have
\begin{eqnarray}
\dsum_{J \subset \{ 1,2, \ldots, n\} }
(-1)^{\sharp J}
\natural (\Delta_{\{ 0\} \sqcup J})
& = &
\dsum_{J \subset \{ 1,2, \ldots, n\} }
(-1)^{n- \sharp J}
\sharp (\Delta_{\{ 0\} \sqcup J})
\\
& = &
\sharp (\hat{\Delta}_1 + \cdots + \hat{\Delta}_n+p).
\end{eqnarray}
Moreover we can easily see that the last
term is equal to $MV(\Delta_1,
\ldots, \Delta_n)$.
\par \noindent (C) Now we consider the general case.
For a polytope $\Box$ in $\RR^n$ let $\1_{\Box}:
\RR^n \longrightarrow \{ 0,1 \}$ (resp.
$\rho_{\Box}: \RR^n \longrightarrow \{ 0, \pm 1 \}$)
be the characteristic function of $\Box$
(resp. the function defined by $\rho_{\Box}=
(-1)^{\d \Box} \1_{\relint (\Box)}$). In particular,
for any point $p \in \RR^n$ we have $\rho_{ \{ p\} }
=\1_{ \{ p\} }$. If $\Box$ and $\Box'$ are polytopes
in $\RR^n$ and $\Box$ majorizes $\Box'$, then for
a face $\Gamma \prec \Box$ of $\Box$ we denote by $\Gamma'$
the corresponding face of $\Box'$. The proof
of the following lemma is very easy and
left to the reader.
\begin{lemma}
In the situation as above, we have
\begin{equation}
\dsum_{\Gamma \prec \Box }
(-1)^{\d \Gamma} \rho_{\Gamma'}=\1_{\Box'},
 \ \qquad \
\dsum_{\Gamma \prec \Box }
(-1)^{\d \Gamma} \1_{\Gamma'}=\rho_{\Box'}.
\end{equation}
\end{lemma}
Actually, we need this lemma in the following
special setting.
\begin{lemma}\label{kel}
Let $\Box$ and $\Box'$ be as above and
$l \subset \RR^n$ a closed ray (i.e. a
closed half segment $\simeq [0, \infty )$)
in $\RR^n$ whose extremal point is the
origin $0 \in \RR^n$. Then we have
\begin{equation}
\dsum_{\Gamma }
(-1)^{\d \Gamma +1} \rho_{\Gamma'+l}
=\1_{\Box'+l},
 \ \qquad \
\dsum_{\Gamma }
(-1)^{\d \Gamma +1} \1_{\Gamma'+l}
=\rho_{\Box'+l},
\end{equation}
where $\Gamma$ ranges through the bounded
faces of $\Box +l$ (they are also faces of
$\Box$) and $\Gamma'$ is the
face of $\Box'$ which corresponds to $\Gamma
\prec \Box$.
\end{lemma}
Now we return to the proof of Proposition
\ref{AE}. Let $f_0, f_1, \ldots, f_n: \RR^n
\longrightarrow \RR$ be polynomials of
order $\leq 1$ such that $f_j|_{\Delta_j}>0$.
For $0 \leq j \leq n$ let $\tl{\Delta}_j
\subset \RR^n \oplus \RR^1$ be the graph of
$f_j|_{\Delta_j}$. For a subset $I \subset
\{ 0,1, \ldots, n \}$ set $\tl{\Delta}_I=
\sum_{j \in I}\tl{\Delta}_j$ and let $l$ be
the closed ray $\{ 0 \} \times \{ x \in \RR \ | \
x \leq 0 \}$ in $\RR^n \oplus \RR^1$.
Then for any $I \subset
\{ 0,1, \ldots, n \}$ the Minkowski sum
$\tl{\Delta}_{\{ 0,1, \ldots, n \} }$ majorizes
the one $\tl{\Delta}_I$. For a face
$\tl{\Gamma} \prec \tl{\Delta}_{\{ 0,1, \ldots, n \} }$
we denote by $\tl{\Gamma}_I$ the corresponding
face of $\tl{\Delta}_I$ and by $\Gamma_I \subset \RR^n$
the projection of $\tl{\Gamma}_I+l \subset
\RR^n \oplus \RR^1$ to $\RR^n$. Then
we have $\d (\tl{\Gamma}_I+l)=\d \Gamma_I +1$.
Note that for $0 \leq j \leq n$ the projection
$\Gamma_{ \{ j\} }$ is a face of $\Delta_j$,
and we have $\Gamma_I=\sum_{j \in I}
\Gamma_{ \{ j\} }$ for any $I \subset
\{ 0,1, \ldots, n \}$.
Moreover we have the following lemma.
For $0 \leq j \leq n$ let $\LL (\Delta_j)$
be the linear subspace of $\RR^n$ which
is parallel to the affine span of $\Delta_j$.
Denote by $f_j^L: \LL (\Delta_j) \longrightarrow
\RR$ the restriction of the linear part of
$f_j$ to $\LL (\Delta_j)$. Let $S=\oplus_{j=0}^n
\LL (\Delta_j)^*$ be the set of ($n+1$)-tuples
$(g_0, g_1, \ldots, g_n)$ of such linear
functions.
\begin{lemma}
There exists an open dense subset $S_{\gen}$ of
$S$ such that for any $(g_0, g_1, \ldots, g_n)
\in S_{\gen}$ we have: If the polynomials
$f_0, f_1, \ldots, f_n: \RR^n
\longrightarrow \RR$ satisfy $f_j^L=g_j$
($0 \leq j \leq n$) then for any bounded
face $\tl{\Gamma}$ of
$\tl{\Delta}_{\{ 0,1, \ldots,  n \} } +l$,
which is also a
face of $\tl{\Delta}_{\{ 0,1,
\ldots, n \} }$, the corresponding
faces $\Gamma_{ \{ 0\} },
\Gamma_{ \{ 1\} }, \ldots, \Gamma_{ \{ n\} }$
are transversal: $\d
(\sum_{j=0}^n \Gamma_{ \{ j\} })
=\sum_{j=0}^n \d \Gamma_{ \{ j\} }$.
\end{lemma}
\begin{proof}
If some faces $G_0, G_1, \ldots, G_0$ of
$\Delta_0, \Delta_1, \ldots, \Delta_n$
correspond to the same bounded face
$\tl{\Gamma}$ of
$\tl{\Delta}_{\{ 0,1,
\ldots,  n \} } +l$ then there exists
a linear function $f: \RR^n
\longrightarrow \RR$ such that
$f|_{\LL (G_j)}=f_j^L|_{\LL (G_j)}$
for any $0 \leq j \leq n$. For
each such ($n+1$)-tuple $(G_0, \ldots, G_n)$
of faces, this last condition gives a
restriction to
$(f_0^L, \ldots, f_n^L)$ and hence
defines a linear subspace $S(G_0, \ldots, G_n)$
of $S$. Note that if $G_0, \ldots, G_n$ are not
transversal the codimension of
$S(G_0, \ldots, G_n)$ is positive.
So it suffices to set $S_{\gen}$ to
be the complement of the union of
such $S(G_0, \ldots, G_n)$'s.
\end{proof}
By this lemma, after changing
the linear parts of
$f_0, f_1, \ldots, f_n$ slightly,
we may assume that
for any bounded face $\tl{\Gamma}$ of
$\tl{\Delta}_{\{ 0,1, \ldots,  n \} } +l$
the corresponding faces $\Gamma_{ \{ 0\} },
\Gamma_{ \{ 1\} }, \ldots, \Gamma_{ \{ n\} }$
are transversal.
Then by applying Lemma \ref{kel} to
the case $\Box =
\tl{\Delta}_{\{ 0,1, \ldots, n \} }$,
$\Box'= \tl{\Delta}_{ \{ 0 \} \sqcup J}$ ($J \subset
\{ 1, 2, \ldots, n \}$) and the closed ray
$l=\{ 0 \} \times \{ x \in \RR \ | \
x \leq 0 \}$ we obtain
\begin{equation}\label{e-1}
\dsum_{J \subset \{ 1,2, \ldots, n\} }
(-1)^{\sharp J}
\natural (\Delta_{\{ 0\} \sqcup J})
=
\dsum_{\tl{\Gamma}}
(-1)^{\d \tl{\Gamma}}
\dsum_{J \subset \{ 1,2, \ldots, n\} }
(-1)^{\sharp J}
\sharp (\Gamma_{\{ 0\} \sqcup J}),
\end{equation}
where $\tl{\Gamma}$ ranges through the
bounded faces of
$\tl{\Delta}_{\{ 0,1, \ldots, n \} }+l$.
By the transversality of $\Gamma_{ \{ 0\} },
\Gamma_{ \{ 1\} }, \ldots, \Gamma_{ \{ n\} }$
for $\tl{\Gamma} \prec
\tl{\Delta}_{\{ 0,1, \ldots, n \} }$ there
are only the following two cases:
\par \noindent (a) There exists $1 \leq j \leq n$
such that $\d \Gamma_{ \{ j\} }=0$.
\par \noindent (b)  $\Gamma_{ \{ 1\} },
\ldots, \Gamma_{ \{ n\} }$
are linearly independent segments and
$\d \Gamma_{ \{ 0\} }=0$.
\par \noindent In the case (a), by
applying Step (A) to
 $\Gamma_{ \{ 0\} },
\Gamma_{ \{ 1\} }, \ldots, \Gamma_{ \{ n\} }$
(the vertices of $\Gamma_{ \{ 1\} }, \ldots,
\Gamma_{ \{ n\} }$
are those of $\Delta_1, \ldots, \Delta_n$
and hence in $L$) we have
\begin{equation}
\dsum_{J \subset \{ 1,2, \ldots, n\} }
(-1)^{n- \sharp J}
\sharp (\Gamma_{\{ 0\} \sqcup J})=0.
\end{equation}
In particular, this is the case whenever
$\d \tl{\Gamma} <n$. Moreover, in the
case (b), by Step (B) we have
\begin{equation}
\dsum_{J \subset \{ 1,2, \ldots, n\} }
(-1)^{n- \sharp J}
\sharp (\Gamma_{\{ 0\} \sqcup J})
=MV(\Gamma_{ \{ 1\} },
\ldots, \Gamma_{ \{ n\} }).
\end{equation}
Hence we get
\begin{eqnarray}\label{e-2}
 &  &
\dsum_{\tl{\Gamma}}
(-1)^{\d \tl{\Gamma}}
\dsum_{J \subset \{ 1,2, \ldots, n\} }
(-1)^{\sharp J}
\sharp (\Gamma_{\{ 0\} \sqcup J})
\\
 & = &
\dsum_{\tl{\Gamma}, \ \d \tl{\Gamma}=n}
(-1)^{n}
\dsum_{J \subset \{ 1,2, \ldots, n\} }
(-1)^{\sharp J}
\sharp (\Gamma_{\{ 0\} \sqcup J})
\\
 & = &
\dsum_{\tl{\Gamma}, \ \d \tl{\Gamma}=n}
MV(\Gamma_{ \{ 1\} },
\ldots, \Gamma_{ \{ n\} }).
\end{eqnarray}
By reversing the arguments used
to obtain \eqref{e-1} and \eqref{e-2}
in the absense of the $0$-th polytopes $\Delta_0$,
$\Gamma_{ \{ 0\} }$ etc.,
we find that the last term of \eqref{e-2}
is equal to
\begin{equation}
\dsum_{J \subset \{ 1,2, \ldots, n\} }
(-1)^{\sharp J}
\natural (\Delta_{J})=
MV(\Delta_1, \ldots, \Delta_n).
\end{equation}
Similarly we have
\begin{eqnarray}
 &  &
\dsum_{J \subset \{ 1,2, \ldots, n\} }
(-1)^{n- \sharp J}
\sharp (\Delta_{\{ 0\} \sqcup J})
\\
 & = &
\dsum_{\tl{\Gamma}}
(-1)^{\d \tl{\Gamma}}
\dsum_{J \subset \{ 1,2, \ldots, n\} }
(-1)^{n- \sharp J}
\natural (\Gamma_{\{ 0\} \sqcup J})
\\
 & = &
\dsum_{\tl{\Gamma}, \ \d \tl{\Gamma}=n}
\dsum_{J \subset \{ 1,2, \ldots, n\} }
(-1)^{\sharp J}
\natural (\Gamma_{\{ 0\} \sqcup J})
\\
 & = &
\dsum_{\tl{\Gamma}, \ \d \tl{\Gamma}=n}
MV(\Gamma_{ \{ 1\} },
\ldots, \Gamma_{ \{ n\} })=
MV(\Delta_1, \ldots, \Delta_n).
\end{eqnarray}
This completes the proof.
\end{proof}
Now let us return to the situation in
Proposition \ref{EUP} and use
the notations there. Then by applying
Proposition \ref{AE} to the case
$\Delta_1= \cdots = \Delta_n=\Delta$,
$\Delta_0=\{ -w(\alpha ) \}$ and $L=L_{\tau}$
we obtain the following very simple result.
We define a finite
subset $\Lambda \subset \CC$ by $\Lambda
=\{ \tau^{v} \ | \ v \in
\ZZ^n \}$.
\begin{theorem}\label{LP}
In the situation as above, we have
\begin{equation}
\chi(Z^*)_{\alpha}=
\dsum_{p,q}  e^{p,q}(Z^*)_{\alpha}
=\begin{cases}
(-1)^{n-1} \frac{1}{\sharp \Lambda}
\Vol_{\ZZ}(\Delta )& (\alpha \in
\Lambda ), \\
0 & (\alpha \notin \Lambda ),
\end{cases}
\end{equation}
where $\Vol_{\ZZ}(*)\in \ZZ$ is
the normalized $n$-dimensional
volume with respect to the lattice $\ZZ^n$.
\end{theorem}
The following definition will be frequently
used throughout this paper.
\begin{definition}
For a subvariety $Y^*
\subset (\CC^*)^n$ which is invariant by
$l_{\tau} \colon (\CC^*)^n \simto (\CC^*)^n$,
$p,q \geq 0$ and $\alpha \in \CC$ we define
the virtual Betti polynomial $\beta (Y^*)_{\alpha}
=\sum_{i=0}^{+ \infty} \beta_i(Y^*)_{\alpha} \cdot t^i
\in \ZZ [t]$ by $\beta_i(Y^*)_{\alpha}
= \sum_{p+q=i} e^{p,q}(Y^*)_{\alpha} \in \ZZ$.
\end{definition}

\begin{definition}\label{PPP}
Let $\Delta$ be an $n$-dimensional integral
polytope in $(\RR^n, \ZZ^n)$.
\begin{enumerate}
\item (see \cite[Section2.3]{D-K})
We say
that $\Delta$ is prime if for any
vertex $w$ of $\Delta$ the cone $\Con(\Delta,w)$
is generated by a basis of $\RR^n$.
\item We
say that $\Delta$ is pseudo-prime
if for any $1$-dimensional face
$\gamma \prec \Delta$ the number of the
$2$-dimensional faces
$\gamma^{\prime} \prec \Delta$ such that
$\gamma \prec \gamma^{\prime}$ is $n-1$.
\end{enumerate}
\end{definition}
By definition, prime polytopes are pseudo-prime.
Moreover, for a pseudo-prime polytope
$\Delta$ the projective
toric variety $X_{\Delta}$
associated to the dual fan of $\Delta$
is an orbifold outside finitely many
points. This implies that the closure
of a non-degenerate hypersurface
$Z^* \subset (\CC^*)^n$ in $X_{\Delta}$
is quasi-smooth in the sense of \cite{D-K}
and has the Poincar\'e duality.
By \cite[Corollary 2.15]{M-T-4} and the proof of
Theorem \ref{LP} we obtain the
following proposition, which enables us to
rewrite the main results of \cite{M-T-4}
much more simply
in terms of the volumes of polytopes
(see \cite[Theorem 5.9]{M-T-4} etc.).

\begin{proposition}\label{PLP}
In the situation of
Proposition \ref{EUP}, assume moreover that
the $n$-dimensional polytope
$\Delta =NP(g)$ is pseudo-prime. Then for
any $\alpha \in \CC \setminus \{ 1 \}$ and
$r \geq 0$ we have
\begin{equation}
\beta_r(Z^*)_{\alpha}
=(-1)^{n+r}
\sum_{\begin{subarray}{c}
\Gamma \prec \Delta\\ \d \Gamma =r+1
\end{subarray}} \left\{
\sum_{\gamma \prec \Gamma}
(-1)^{\d \gamma }
\frac{1}{\sharp \Lambda (\gamma)}
\Vol_{\ZZ}(\gamma )_{\alpha}
\right\},
\end{equation}
where $\Lambda (\gamma ) \subset \Lambda$
is defined
similarly to $\Lambda$ by using
the intersection of $\ZZ^n$ and
the affine span of $\gamma $, and
we define the integer
$\Vol_{\ZZ}(\gamma )_{\alpha} \in \ZZ_+$ by
\begin{equation}
\Vol_{\ZZ}(\gamma )_{\alpha}
= \begin{cases}
\Vol_{\ZZ}(\gamma )
& (\alpha \in
\Lambda (\gamma ) ), \\
0 & (\alpha \notin \Lambda (\gamma ) ).
\end{cases}
\end{equation}
\end{proposition}

Now let $\Box \subset \RR^n$ be
an ($n-1$)-dimensional
integral polytope whose affine span
$K \simeq \RR^{n-1}$ in $\RR^n$ does
not contain the origin $0 \in \RR^n$.
Denote by $\Delta$ the pyramid over
$\Box$ with apex $0 \in \RR^n$
and let $d_{\Box}>0$ be the lattice distance
of $\Box$ from $0 \in \RR^n$.
Let $Z^* \subset (\CC^*)^n$ be a
non-degenerate hypersurface
whose Newton polytope is $\Delta$.
Assume also that the support of
the defining Laurent polynomial of $Z^*$
is contained in $\{ 0 \} \sqcup \Box$.
Then we can define an automorphism of
$Z^*$ of order $d_{\Box}$ as follows.
Let $\height (* , K): \RR^n \longrightarrow
\RR$ be the linear map such that
$\height (v, K) =d_{\Box}>0$ for any
$v \in K$. Then
to the group homomorphism
$\ZZ^n \longrightarrow \CC^*$ defined by
\begin{equation}
v \longmapsto \exp (2 \pi \sqrt{-1} \cdot
\height (v, K)/d_{\Box})
\end{equation}
\noindent we can naturally
associate an element $\tau_{\Box} \in
(\CC^*)^n=\Spec (\CC [\ZZ^n])$
such that $(\tau_{\Box})^{d_{\Box}}=1$.
By construction
$Z^* \subset (\CC^*)^n$ is invariant
by the multiplication by $\tau_{\Box}$. Now
fix a complex number $\alpha \not= 1$.
Then by Theorem \ref{LP}, the
virtual Betti polynomial
\begin{equation}
\beta (\Box)_{\alpha}:=\beta (Z^*)_{\alpha}
\in \ZZ [t]
\end{equation}
\noindent (of degree $\leq \d \Box =n-1$) of
the hypersurface $Z^* \subset (\CC^*)^n$
defined by the ($n-1$)-dimensional
polytope $\Box$ can be calculated as follows.
First, for each face $\Gamma$ of $\Box$
we define a polynomial $\beta (\Gamma)_{\alpha}
\in \ZZ [t]$ of degree $\leq \d \Gamma$
similarly. By induction on $\d \Box$,
we may assume that for any proper face
$\Gamma$ of $\Box$ the polynomial
$\beta (\Gamma)_{\alpha}$ is already
determined. Let $\Box^{\prime}$ be
an ($n-1$)-dimensional prime integral
polytope which majorizes $\Box$
in the affine span $K$ of $\Box$.
For a face $\Gamma^{\prime}$ of
$\Box^{\prime}$ we denote by $\Gamma$
the corresponding face of $\Box$.
Then we can (uniquely) determine
$\beta (\Box)_{\alpha} \in \ZZ [t]$
by the following three conditions:

\medskip

\par \noindent (i) The degree of
$\beta (\Box)_{\alpha}$ is $\leq \d \Box$.
\par \noindent (ii) The coefficients $c_i$ of
the polynomial
\begin{equation}
\sum_{i=0}^{+ \infty} c_i t^i=
\dsum_{ \Gamma^{\prime}
\prec \Box^{\prime}}
(t^2-1)^{\d \Gamma^{\prime}- \d \Gamma}
\beta (\Gamma)_{\alpha} \in \ZZ [t]
\end{equation}
are symmetric with respect to the degree
$\d \Box$: $c_{\d \Box +k}=c_{\d \Box -k}$ for
any $k \in \ZZ$.
\par \noindent (iii) $\beta (\Box)_{\alpha}(1)
=(-1)^{\d \Box}
\Vol_{\ZZ}(\Box)_{\alpha}$, where we set
\begin{equation}
\Vol_{\ZZ}(\Box )_{\alpha}
= \begin{cases}
\Vol_{\ZZ}(\Box ) \in \ZZ_+
& ( \alpha^{d_{\Box}}=1 ), \\
0 & (\text{otherwise} ).
\end{cases}
\end{equation}
Indeed, let $\Delta^{\prime}$ be the
pyramid over $\Box^{\prime}$
 with apex $0 \in \RR^n$ and
$X_{\Delta^{\prime}}$ the projective
toric variety associated to its dual fan.
Note that $\Delta^{\prime}$ is pseudo-prime
and majorizes $\Delta$.
Then the closure of
$Z^* \subset (\CC^*)^n$ in $X_{\Delta^{\prime}}$
has the Poincar\'e duality, and
(as in \cite[Section 5.2]{D-K})
by using Theorem \ref{LP}
we obtain the above algorithm for
the computation of $\beta (\Box)_{\alpha}$.
The following definition will play a
crucial role in the proof of our main results.

\begin{definition}\label{AB}
For a complex number $\alpha \not= 1$
let $\beta (\Box)_{\alpha}\in \ZZ [t]$
be the polynomial of degree $\leq \d \Box$
as above. Then for $m \geq \d \Box$
we set
\begin{equation}
\beta (\Box, m)_{\alpha}=
(t^2-1)^{m- \d \Box}
\beta (\Box)_{\alpha} \in \ZZ [t].
\end{equation}
\end{definition}

\section{Motivic Milnor fibers over C.I. and their
virtual Betti polynomials}\label{sec:4}

For $2 \leq k \leq n$ let
\begin{equation}
W= \{ f_1= \cdots =f_{k-1}=0 \} \supset
V= \{ f_1= \cdots =f_{k-1}=f_k=0 \}
\end{equation}
be complete intersection subvarieties
of $\CC^n$ such that $0 \in V$. Assume
that $W$ and $V$ have
isolated singularities at the origin
$0 \in \CC^n$. Then by a fundamental
result of Hamm \cite{Hamm} the Milnor
fiber $F_0$ of $g:=f_k|_W \colon W
\longrightarrow \CC$ at the origin $0$
satisfies the condition $H^j(F_0;\CC) \simeq 0$
($j\neq 0, \ n-k$). Recall that the
semisimple part of the monodromy
operator $\Phi_{n-k,0}\colon H^{n-k} (F_0;\CC)
\simeq  H^{n-k} (F_0;\CC)$
was determined by Oka \cite{Oka-3}, \cite{Oka-2}
and Kirillov \cite{Kirillov} (see also
\cite{M-T-2} for some generalizations).
Our objective here is
to describe the Jordan normal form of $\Phi_{n-k,0}:
H^{n-k} (F_0;\CC) \simeq  H^{n-k} (F_0;\CC)$
in terms of the Newton
polyhedrons of $f_1, f_2, \ldots , f_k$.
For this purpose, we shall use the theory of
mixed Hodge modules due to Saito \cite{Saito-1} and
\cite{Saito-2}.
Let $\psi_{f_k}^p:=\psi_{f_k}[-1] :
\Dbc(\CC^n) \longrightarrow \Dbc(f_k^{-1}(0))$
be the shifted nearby cycle functor which
preserves the perversity. Let $\F \in
\Dbc(\CC^n)$ be the minimal extension of
the perverse sheaf $\CC_{W \setminus
\{ 0\}}[n-k+1] \in \Dbc(\CC^n \setminus
\{ 0\})$ to $\CC^n$. Then the perverse
sheaf $\psi_{f_k}^p (\F ) \in
\Dbc(f_k^{-1}(0))$
on $f_k^{-1}(0)$ has the following
decomposition with respect to the
eigenvalues $\lambda \in \CC^*$ of
its monodromy automorphism:
\begin{equation}
\psi_{f_k}^p (\F )= \bigoplus_{\lambda \in \CC^*}
\psi_{f_k, \lambda}^p (\F )
\end{equation}
(see \cite{Dimca} etc.). By Proposition
\ref{prp:2-7}
for any $\lambda \not= 1$ the support of
the perverse sheaf $\psi_{f_k, \lambda}^p (\F )$
is contained in the origin $0 \in \CC^n$.
So we may regard $\psi_{f_k, \lambda}^p (\F )$
($\lambda \not= 1$) simply as complex vector spaces
endowed with monodromy automorphisms.
Now by using
the mixed Hodge module over the
perverse sheaf $\F \in \Dbc(\CC^n)$, to
$\psi_{f_k}^p(\F )_0 \in \Dbc(\{ 0 \})$
and the semisimple part of its
monodromy automorphism, we
associate naturally an element
\begin{equation}
[H_g] \in \KK_0(\HSm)
\end{equation}
(see Saito \cite{Saito-1} and
\cite{Saito-2} for the details).
Then by construction, for any $\lambda
\not= 1$ the $\lambda$-eigenspace part
$[H_g]_{\lambda} \in \KK_0(\HS )$
of $[H_g] \in \KK_0(\HSm)$ is
identified with the complex
vector space $\psi_{f_k, \lambda}^p (\F )$
endowed with
a Hodge decomposition whose weights are
defined by its ``absolute" monodromy
filtration (see Saito \cite{Saito-1} and
\cite{Saito-2}). Here we
essentially used the purity of
the mixed Hodge module over the
perverse sheaf $\F \in \Dbc(\CC^n)$.
For an element $[H] \in \KK_0(\HSm)$,
$H \in \HSm $ with a quasi-unipotent
endomorphism $\Psi : H \simto
H$, $p, q \geq 0$ and $\lambda \in \CC$
denote by $e^{p,q}([H])_{\lambda}$
the dimension of the $\lambda$-eigenspace
of the morphism $H^{p,q} \simto
H^{p,q}$ induced by $\Psi$ on the
$(p,q)$-part $H^{p,q}$ of $H$. Then the
following results are immediate consequences
of the above construction
and Saito's very deep theory
in \cite{Saito-1} and \cite{Saito-2}.
Indeed, we can check the assertion (i) below
by explicitly calculating the mixed Hodge
numbers of our motivic Milnor fiber $\SS_{g,0}$.

\begin{proposition}\label{NJB}
Assume that $\lambda \in \CC^* \setminus \{1\}$. Then
\begin{enumerate}
\item  We have $e^{p,q}(
[H_g])_{\lambda}=0$ for $(p,q)
\notin [0,n-k] \times [0,n-k]$.
Moreover for $(p,q) \in [0,n-k] \times [0,n-k]$ we have
\begin{equation}
e^{p,q}( [H_g])_{\lambda}=e^{n-k-q,n-k-p}(
[H_g])_{\lambda}.
\end{equation}
\item For $i \geq 1$, the number
of the Jordan blocks
for the eigenvalue $\lambda$ with sizes $\geq i$
in $\Phi_{n-k, 0} : H^{n-k}(F_0;\CC) \simto
H^{n-k}(F_0;\CC)$ is equal to
\begin{equation}
 \sum_{p+q=n-k-1+i, n-k+i} e^{p,q}(
[H_g])_{\lambda}.
\end{equation}
\end{enumerate}
\end{proposition}

\begin{remark}\label{EV1}
By Proposition \ref{prp:2-7},
for $\lambda =1$ the dimension of the support
of $\psi_{f_k, \lambda}^p (\F )$ is not
zero in general. Therefore
for $\lambda =1$ we cannot prove
the symmetry of weights of
$[H_g]_{\lambda} \in \KK_0(\HS )$ as in
Proposition \ref{NJB} (i) (indeed we
can easily find counterexamples).
This fact explains the reason why
the results on the Jordan blocks for the
eigenvalue $1$ in $\Phi_{n-k,0}$
cannot be obtained by our methods. For
related problems, see also for
example Ebeling-Steenbrink \cite{E-S}.
\end{remark}

By Proposition \ref{NJB},
for $\lambda \not= 1$ the calculation of
the eigenvalue $\lambda$ part of the
Jordan normal form of $\Phi_{n-k,0}$ is
reduced to that of $e^{p,q}([H_g])_{\lambda}$.
Moreover, as in Denef-Loeser \cite{D-L-1},
\cite{D-L-2} and
Guibert-Loeser-Merle \cite{G-L-M},
by using a resolution of
singularities of $W$ and $g : W
\longrightarrow \CC$ we can
construct a motivic Milnor fiber $\SS_{g,0}$
of $g$ at $0 \in \CC^n$
which enables us to calculate
$e^{p,q}([H_g])_{\lambda}$
as follows. Let $\pi : X \longrightarrow \CC^n$
be a proper morphism from a
smooth algebraic variety $X$ such that
$\pi |_{X \setminus \pi^{-1}(0)}:
X \setminus \pi^{-1}(0) \longrightarrow
\CC^n \setminus \{ 0\}$ is an isomorphism and
$\pi^{-1}(0)=D_1\cup \cdots \cup D_m$ is a normal crossing
divisor ($D_1,\ldots, D_m$ are smooth) in $X$.
Then via the isomorphism $X \setminus \pi^{-1}(0)
\simeq \CC^n \setminus \{ 0\}$ we regard
$W \setminus \{ 0\}$ as a subset of $X$ and
denote by $W^{\prime}$ its closure in $X$.
We call $W^{\prime}$ the proper transform of
$W$ in $X$. By Hironaka's theorem we can take
$\pi : X \longrightarrow \CC^n$ such that
$W^{\prime}$ is smooth and intersects
$D_I:=\bigcap_{i \in I}D_i$ transversally
for any subset $I \subset \{ 1,2,
\ldots , m \}$. We may assume also that the
hypersurface
$S:=\overline{f_k^{-1}(0) \setminus \{ 0 \}}
\subset X$ is smooth in a neighborhood of
$W^{\prime}$ and intersects
$D_I \cap W^{\prime}$ transversally
for any $I \subset \{ 1,2,
\ldots , m \}$. For $1 \leq i \leq m$ let
$d_i>0$ be the order of the zero of
$f_k \circ \pi$ along $D_i$. For a non-empty subset
$I \subset \{ 1,2, \ldots , m \}$ set
$D_I^{\circ}=D_I \setminus (\bigcup_{i \notin I}D_i)$,
\begin{equation}
E_I^{\circ}=(D_I^{\circ} \cap W^{\prime})
\setminus S, \hspace{10mm}
F_I^{\circ}=(D_I^{\circ} \cap W^{\prime}) \cap S
\end{equation}
and $d_I={\rm gcd} (d_i)_{i \in I}>0$.
Then, as in \cite[Section 3.3]{D-L-2}, we
can construct an unramified
Galois covering $\tl{E_I^{\circ}}
\longrightarrow E_I^{\circ}$ of $E_I^{\circ}$
as follows. First, let $U
\subset X \setminus S$ be an affine
open subset such that $f_k \circ \pi = h_{1,W}
 \cdot (h_{2,W})^{d_I}$ on $U$, where $ h_{1,W}$
is a unit on $U$ and
$h_{2,W} : U \longrightarrow \CC$
is a regular function. It is
easy to see that $E_I^{\circ}$
is covered by such open subsets $U$.
Then by gluing the varieties
\begin{equation}
\{(t,x) \in \CC^* \times (E_I^{\circ} \cap U)
\ |\  h_{1,W}(x) \cdot t^{d_I} -1=0 \}
\end{equation}
together in an obviously way we obtain
the $d_I$-fold covering $\tl{E_I^{\circ}}$ of
$E_I^{\circ}$. Now for $d \in \ZZ_{>0}$,
let $\mu_d \simeq \ZZ/\ZZ d$ be the
multiplicative group consisting of the
$d$ roots of unity
$\{ 1, \zeta_{d}, \zeta_{d}^2, \ldots, \zeta_{d}^{d-1}
 \}$, where we set $\zeta_{d}:=
\exp (2\pi\sqrt{-1}/d) \in \CC$.
Then the unramified
Galois covering $\tl{E_I^{\circ}}$
of $E_I^{\circ}$ admits a natural action of
$\mu_{d_I}$ defined by
assigning the automorphism $(t,x) \longmapsto
(\zeta_{d_I} t, x)$ of $\tl{E_I^{\circ}}$
to the generator
$\zeta_{d_I} \in \mu_{d_I}$. Moreover, let
$\hat{\mu}$ be the projective limit
$\underset{d}{\varprojlim} \mu_d$ of the
projective system $\{ \mu_i \}_{i \geq 1}$
with morphisms $\mu_{id}
\longrightarrow \mu_i$ given by
$t \longmapsto t^d$.  Then the variety
$\tl{E_I^{\circ}}$ is endowed with
a good $\hat{\mu}$-action in the sense
of \cite[Section 2.4]{D-L-2}.
Following the notations in \cite{D-L-2},
denote by $\M_{\CC}^{\hat{\mu}}$ the
ring obtained from the Grothendieck
ring $\KK_0^{\hat{\mu}}(\Var_{\CC})$ of
varieties over $\CC$ with good
$\hat{\mu}$-actions by
inverting the Lefschetz motive $\LL\simeq \CC \in
\KK_0^{\hat{\mu}}(\Var_{\CC})$. Recall that $\LL \in
\KK_0^{\hat{\mu}}(\Var_{\CC})$
is endowed with the trivial action of
$\hat{\mu}$. We denote by $[\tl{E_I^{\circ}}]$
(resp. $[ F_I^{\circ} ]$)
the class of the variety $\tl{E_I^{\circ}}$
(resp. $F_I^{\circ}$) endowed
with the above $\hat{\mu}$-action
(resp. the trivial $\hat{\mu}$-action) in
$\M_{\CC}^{\hat{\mu}}$.

\begin{definition}
(\cite{D-L-1},
\cite{D-L-2} and \cite{G-L-M})
We define the motivic Milnor
fiber $\SS_{g,0} \in \M_{\CC}^{\hat{\mu}}$ of $g : W
\longrightarrow \CC$ at the origin $0 \in \CC^n$ by
\begin{equation}
\SS_{g,0} =\sum_{I \neq \emptyset}
\left\{ (1-\LL)^{\sharp I -1}
[\tl{E_I^{\circ}}] + (1-\LL)^{\sharp I}
[F_I^{\circ}] \right\} \in \M_{\CC}^{\hat{\mu}}.
\end{equation}
\end{definition}

For the description of the element
$[H_g]\in \KK_0(\HSm)$ in terms of
$\SS_{g,0} \in \M_{\CC}^{\hat{\mu}}$, let
\begin{equation}
\chi_h \colon \M_{\CC}^{\hat{\mu}}
\longrightarrow \KK_0(\HSm)
\end{equation}
be the Hodge characteristic map
defined in \cite{D-L-2}.  To a variety $Z$ with
a good $\mu_d$-action it
associates the Hodge structure
\begin{equation}
\chi_h ([Z])=\sum_{j \in \ZZ} (-1)^j
[H_c^j(Z;\QQ)] \in \KK_0(\HSm)
\end{equation}
with the actions induced by the one
$z \longmapsto \zeta_d \cdot z$
($z\in Z$) on $Z$.
Then by the proof of Denef-Loeser
\cite[Theorem 4.2.1]{D-L-1}
we obtain the following result.

\begin{theorem}\label{ITM}
In the Grothendieck group $\KK_0(\HSm)$, we have
\begin{equation}
[H_g]= (-1)^{n-k} \chi_h(\SS_{g,0}).
\end{equation}
\end{theorem}

Thus our problem was reduced to the
calculation of $\chi_h(\SS_{g,0})
\in \KK_0(\HSm)$.

\begin{definition}
Let $f(x) \in \CC[x_1,\ldots,x_n]$ be
a polynomial on $\CC^n$.
\begin{enumerate}
\item We call the convex hull of
$\bigcup_{v \in \supp f} \{ v+\RR_+^n \}$ in
$\RR_+^n$ the Newton polyhedron
of $f$ at the origin $0 \in \CC^n$
and denote it by $\Gamma_+(f)$.
\item We say that $f$ is convenient if $\Gamma_+(f)$
intersects each coordinate axis of $\RR^n$
outside the origin.
\end{enumerate}
\end{definition}

From now on, in order to describe our
results explicitly, assume also that $f_1,
f_2, \ldots , f_k$ are convenient.
Set $f:=(f_1, f_2, \ldots , f_k)$ and
\begin{equation}
\Gamma_+(f):=\Gamma_+(f_1)+\Gamma_+(f_2)+
\cdots + \Gamma_+(f_k).
\end{equation}
We denote the union of compact faces of
$\Gamma_+(f)$ by $\Gamma_f$. Recall
that on $\RR_+^n$ we can define an equivalence
relation by $u \sim
u^{\prime}$ $\Longleftrightarrow$ the
supporting faces of $u$ and
$u^{\prime}$ in $\Gamma_+(f)$ are the same.
Then we obtain a decomposition
$\RR^n_+=\bigsqcup_{\Theta \prec \Gamma_+(f)}
\sigma_{\Theta}$ of $\RR^n_+$
into locally closed cones $\sigma_{\Theta}$.
Since for a face $\Theta \prec
\Gamma_+(f)$ such that $\Theta \subset
\Gamma_f$ (i.e. a compact face $\Theta$ of
$\Gamma_+(f)$)
the supporting face of
$u \in \sigma_{\Theta}$ in
$\Gamma_+(f_j)$ does not depend on the choice
of $u \in \sigma_{\Theta}$, we
denote it simply by $\gamma_j^{\Theta}$.
Then we have
\begin{equation}
\Theta = \gamma_1^{\Theta} + \gamma_2^{\Theta}
+ \cdots + \gamma_k^{\Theta}.
\end{equation}
For $1 \leq j \leq k$ and a compact face
$\Theta$ of $\Gamma_+(f)$ we set
\begin{equation}
f_j^{\Theta}(x)=\sum_{v \in \gamma_j^{\Theta} \cap
\ZZ_+^n} a_vx^v \in \CC [x_1, x_2, \ldots, x_n],
\end{equation}
where $f_j(x)=\sum_{v\in \ZZ_+^n}
a_vx^v$ ($a_v\in \CC$).
\begin{definition}[see \cite{Oka-2} etc.]\label{NDCI}
We say that $f=(f_1, f_2, \ldots , f_k)$ is
non-degenerate at the origin
$0 \in \CC^n$ if for
any compact face $\Theta$ of $\Gamma_+(f)$
the two subvarieties $\{ f_1^{\Theta}(x)= \cdots
=f_{k-1}^{\Theta}(x)=0 \}$ and
$\{ f_1^{\Theta}(x)= \cdots
=f_{k-1}^{\Theta}(x)=
f_{k}^{\Theta}(x)=0 \}$ in $(\CC^*)^n$
are non-degenerate complete intersections.
\end{definition}
From now on, let us assume also that $f$ is
non-degenerate at the origin $0 \in \CC^n$.
Then we can construct the morphism
 $\pi : X \longrightarrow \CC^n$ explicitly
as follows. Let $\Sigma_1=\{
\overline{\sigma_{\Theta}}
\}_{\Theta \prec \Gamma_+(f)}$ be the dual
fan of $\Gamma_+(f)$. Take a smooth
subdivision $\Sigma$ of $\Sigma_1$ and
denote by $X_{\Sigma}$ the smooth toric variety
associated to the (smooth) fan $\Sigma$.
We thus obtain a proper morphism
$\pi : X_{\Sigma} \longrightarrow \CC^n$
which induces an isomorphism
$X_{\Sigma} \setminus \pi^{-1}(0)
\simeq \CC^n \setminus \{ 0\}$. Let
$\rho_1, \rho_2, \ldots, \rho_m$ be the
$1$-dimensional cones in $\Sigma$ such
that $\rho_i \setminus \{ 0 \} \subset
\Int (\RR^n_+)$ and for each
$1 \leq i \leq m$ denote by $D_i$
the (smooth) toric divisor in $X_{\Sigma}$
which corresponds to $\rho_i$.
Then we have $\pi^{-1}(0)=
D_1\cup \cdots \cup D_m$ and
it is a normal crossing
divisor in $X_{\Sigma}$. Moreover,
by the non-degeneracy of $f$,
the proper transforms $W^{\prime}$,
$S=\overline{f_k^{-1}(0) \setminus \{ 0 \}}$
and $D_i$'s satisfy
the required smoothness and transversality.
By using this explicit construction of
$\pi : X_{\Sigma} \longrightarrow \CC^n$
we can express the Hodge realizations
$\chi_h(\SS_{g,0})_{\lambda} \in \KK_0(\HS )$
($\lambda \not= 1$) of our motivic
Milnor fiber $\SS_{g,0}$ very concretely
as follows.
For a face $\Theta \prec \Gamma_+(f)$ such
that $\Theta \subset \Gamma_f$ let
$\LL_{\Theta}\simeq \RR^{\d \Theta}$
be the linear subspace of  $\RR^n$ which is parallel
to the affine span of $\Theta$. We denote by
$K_{\Theta} \simeq \RR^{\d \Theta}$ the affine linear
subspace of $\RR^n$ which is parallel to
$\LL_{\Theta}$ and contains $\gamma_k^{\Theta}$.
Let $\tl{\LL}_{\Theta} \simeq \RR^{\d \Theta +1}$
be the linear subspace of  $\RR^n$
generated by $\{ 0\} \sqcup K_{\Theta}$.
Then $\LL_{\Theta}$ is a hyperplane of
$\tl{\LL}_{\Theta}$, and to the lattice
$\tl{M}_{\Theta}=\ZZ^n \cap \tl{\LL}_{\Theta}$
we can naturally associate the algebraic torus
\begin{equation}
\tl{T}_{\Theta} =
\Spec (\CC [\tl{M}_{\Theta}])
\simeq (\CC^*)^{\d \Theta +1}.
\end{equation}
Denote the convex hull of $\{ 0\}
\sqcup \gamma_k^{\Theta}$ in $\tl{\LL}_{\Theta}$
by $\Delta_{\gamma_k^{\Theta}}$ and for
$1 \leq j \leq k-1$ let $\kappa_j^{\Theta}$ be
an integral translation of $\gamma_j^{\Theta}$
in $\tl{\LL}_{\Theta}$ such that $\kappa_j^{\Theta}
\subset K_{\Theta}$.
For simplicity, we denote
the $k$-tuple $(\kappa_1^{\Theta}, \ldots,
\kappa_{k-1}^{\Theta}, \Delta_{\gamma_k^{\Theta}})$
of integral polytopes in $(\tl{\LL}_{\Theta},
\tl{M}_{\Theta})$ by $\Delta_{\Theta}$.
Let $d_{\Theta} >0$ be the lattice
distance of the hyperplane $K_{\Theta} \subset
\tl{\LL}_{\Theta}$
from the origin $0 \in \tl{\LL}_{\Theta}$. Note
that $d_{\Theta}$ can be an integral multiple of
the lattice distance $d(\gamma_k^{\Theta})>0$ of
$\gamma_k^{\Theta}$ from $0
\in \tl{\LL}_{\Theta}$ if $\d
\gamma_k^{\Theta} < \d \Theta$. Then to $\Delta_{\Theta}$
we can naturally associate a non-degenerate
complete intersection subvariety $Z^*_{\Delta_{\Theta}}$
of $\tl{T}_{\Theta} \simeq (\CC^*)^{\d \Theta +1}$
and an action of the cyclic group
$\mu_{d_{\Theta}}=\ZZ/\ZZ d_{\Theta}$
on it as follows. Let $g_j^{\Theta}$
($j=1,2, \ldots , k-1$) and $\tl{g}_k^{\Theta}$
be Laurent polynomials on
$\tl{T}_{\Theta}$ such that $NP(g_j^{\Theta})=
\kappa_j^{\Theta}$ and $NP(\tl{g}_k^{\Theta})=
\Delta_{\gamma_k^{\Theta}}$. Assume also that
the support $\supp \tl{g}_k^{\Theta}$ of
$\tl{g}_k^{\Theta}$ is contained in
$\{ 0\} \sqcup \gamma_k^{\Theta}$ and the
subvariety
\begin{equation}
Z^*_{\Delta_{\Theta}}=\{ g_1^{\Theta}(x) = \cdots
=g_{k-1}^{\Theta}(x)
=\tl{g}_k^{\Theta}(x)=0 \} \subset \tl{T}_{\Theta}
\end{equation}
of $\tl{T}_{\Theta}$ is a non-degenerate
complete intersection.
Let $\height (* , K_{\Theta}): \tl{\LL}_{\Theta}
\longrightarrow \RR$ be the linear map such that
$\height (v, K_{\Theta}) =d_{\Theta}>0$ for any
$v \in K_{\Theta}$. Then
to the group homomorphism
$\tl{M}_{\Theta} \longrightarrow \CC^*$ defined by
\begin{equation}
v \longmapsto \exp (2 \pi \sqrt{-1} \cdot
\height (v, K_{\Theta})/d_{\Theta})
\end{equation}
\noindent we can naturally
associate an element $\tau_{\Theta} \in
\tl{T}_{\Theta}=\Spec (\CC [\tl{M}_{\Theta}])$
such that $(\tau_{\Theta})^{d_{\Theta}}=1$.
Since $Z^*_{\Delta_{\Theta}} \subset
\tl{T}_{\Theta}$ is invariant by the multiplication
$l_{\tau_{\Theta} }
\colon  \tl{T}_{\Theta} \simto
\tl{T}_{\Theta}$ by $\tau_{\Theta}$, the variety
$Z^*_{\Delta_{\Theta}}$ admits an
action of $\mu_{d_{\Theta}}$. We thus
obtain an element $[Z^*_{\Delta_{\Theta}}]
\in \M_{\CC}^{\hat{\mu}}$.
Finally, for the compact face
$\Theta \prec \Gamma_+(f)$, let
$s_{\Theta}$ be the dimension of the
minimal coordinate subspace of $\RR^n$ containing
$\Theta$ and set
$m_{\Theta}=s_{\Theta}-\d \Theta -1\geq 0$.

\begin{theorem}\label{general}
Assume that $\lambda \in \CC^*\setminus \{1\}$. Then
\begin{enumerate}
\item In the
Grothendieck group $\KK_0(\HS )$, we have
\begin{equation}
\chi_h(\SS_{g,0})_{\lambda}= \dsum_{
\Theta \subset \Gamma_f,
 \d \Theta \geq k-1 }
\chi_h\big((1-\LL)^{m_{\Theta}}
\cdot[Z_{\Delta_{\Theta}}^*]\big)_{\lambda}.
\end{equation}
In particular, the virtual Betti polynomial
$\beta (\SS_{g,0})_{\lambda} \in \ZZ [t]$ is given by
\begin{equation}
\beta (\SS_{g,0})_{\lambda}= \dsum_{
\Theta \subset \Gamma_f,
 \d \Theta \geq k-1 }
(1-t^2)^{m_{\Theta}} \cdot
\beta (Z^*_{\Delta_{\Theta}})_{\lambda}.
\end{equation}
\item For $i \geq 1$, the number
of the Jordan blocks for the eigenvalue
$\lambda$ with sizes $\geq i$ in
$\Phi_{n-k,0}\colon H^{n-k}(F_0;\CC) \simeq
H^{n-k}(F_0;\CC)$ is equal to
\begin{equation}
(-1)^{n-k} \left\{ \beta_{n-k-1+i}(\SS_{g,0})_{\lambda}
+ \beta_{n-k+i}(\SS_{g,0})_{\lambda} \right\}.
\end{equation}
\end{enumerate}
\end{theorem}
\begin{proof}
By using the above explicit construction of
$\pi : X_{\Sigma} \longrightarrow \CC^n$
from $\Gamma_+(f)$ the proof of (i) is
obtained completely similarly to that of
\cite[Theorems 5.3 and 7.3]{M-T-4}. Then the assertion
(ii) follows immediately from
Proposition \ref{NJB} (ii).
\end{proof}

By the Cayley trick in
\cite[Section 6]{D-K} we can rewrite the formula
for $\beta (\SS_{g,0})_{\lambda} \in \ZZ [t]$
($\lambda \not= 1$) in
Theorem \ref{general} (i) as follows.
For a face $\Theta \prec \Gamma_+(f)$ such
that $\Theta \subset \Gamma_f$ we define
an open subset $\Omega_{\Theta}$ of
$\tl{T}_{\Theta} \times \PP^{k-1}$ by
\begin{equation}
\Omega_{\Theta}=\{ (x; (\alpha_1: \cdots :
\alpha_k)) \in \tl{T}_{\Theta} \times \PP^{k-1}
 \ | \ \sum_{j=1}^{k-1}\alpha_j g_j^{\Theta}(x)
+\alpha_k \tl{g}_k^{\Theta}(x) \not= 0 \}.
\end{equation}
By the standard decomposition $\CC^k=
\bigsqcup_{I \subset
\{ 1,2, \ldots, k\}} T_I$, $T_I \simeq
(\CC^*)^{\sharp I}$
of $\CC^k$, we obtain a stratification
$\PP^{k-1}=
\bigsqcup_{I \not= \emptyset} \PP (T_I)$
of $\PP^{k-1}$,
where we set
\begin{equation}
\PP (T_I)=\{ (\alpha_1: \cdots :
\alpha_k) \in \PP^{k-1}
\ | \ \alpha_j=0 \ (j \notin I), \
\alpha_j \not= 0 \ (j \in I) \} \simeq
(\CC^*)^{\sharp I-1}.
\end{equation}
For each subset $J \subset \{ 1,2, \ldots, k-1\}$
($J$ can be an empty set $\emptyset$), set
\begin{equation}
\Omega_{\Theta, J}=\left\{ \tl{T}_{\Theta} \times
\PP (T_{J \sqcup \{ k\}})\right\}
\cap \Omega_{\Theta}.
\end{equation}
Note that $\Omega_{\Theta, J}$ is the complement
of the hypersurface
\begin{equation}
Z^*_{\Theta, J}=\{ (x; \alpha_j
(j \in J)) \in \tl{T}_{\Theta} \times
(\CC^*)^{\sharp J}
\ | \ \sum_{j \in J} \alpha_j g_j^{\Theta}(x)
+ \tl{g}_k^{\Theta}(x) = 0 \}
\end{equation}
of the algebraic torus $\tl{T}_{\Theta} \times
(\CC^*)^{\sharp J}$. Since this hypersurface
$Z^*_{\Theta, J}$ is invariant by the multiplication
of $(\tau_{\Theta}, 1 ) \in \tl{T}_{\Theta} \times
(\CC^*)^{\sharp J}$, we obtain an action of
$\mu_{d_{\Theta}}$ on $Z^*_{\Theta, J}$ and
an element $[Z^*_{\Theta, J}] \in
\M_{\CC}^{\hat{\mu}}$.
\begin{lemma}
For $\lambda \in \CC^*\setminus\{ 1\}$ and
 a face $\Theta \prec \Gamma_+(f)$ such
that $\Theta \subset \Gamma_f$ we have
\begin{equation}
\beta (Z^*_{\Delta_{\Theta}})_{\lambda}=
\frac{1}{t^{2k-2}} \sum_{J \subset
\{ 1,2, \ldots, k-1\}}
\beta ( Z^*_{\Theta, J} )_{\lambda}.
\end{equation}
\end{lemma}
\begin{proof}
By the definition of $Z^*_{\Delta_{\Theta}}$
the natural projection
\begin{equation}
\pi : \Omega_{\Theta} \longrightarrow
\tl{T}_{\Theta} \setminus Z^*_{\Delta_{\Theta}}
\end{equation}
is an algebraic fiber bundle whose fiber is isomorphic
to $\CC^{k-1}$. Therefore by the condition
$\lambda \not= 1$ we obtain
\begin{equation}
\beta (\tl{T}_{\Theta} \setminus
Z^*_{\Delta_{\Theta}})_{\lambda}=
\frac{1}{t^{2k-2}} \sum_{J \subset
\{ 1,2, \ldots, k-1\}}
\beta (\Omega_{\Theta, J})_{\lambda}.
\end{equation}
Since the multiplication of $\tau_{\Theta}$
on $\tl{T}_{\Theta}$ (resp. $(\tau_{\Theta}, 1 )$
on $\tl{T}_{\Theta} \times
(\CC^*)^{\sharp J}$) is homotopic to the identity,
by $\lambda \not= 1$
we obtain the desired formula.
This completes the proof.
\end{proof}
First, by replacing the defining equations of
$Z^*_{\Delta_{\Theta}}$ with the help of
Lemma \ref{INV}, we may assume that
the hypersurfaces $Z^*_{\Theta, J}
\subset \tl{T}_{\Theta} \times
(\CC^*)^{\sharp J}$ are
non-degenerate.
For a face $\Theta \prec \Gamma_+(f)$ such
that $\Theta \subset \Gamma_f$ and
$J \subset \{ 1,2, \ldots, k-1\}$ we denote by
$\Theta_J$ the convex hull of
\begin{equation}
(\gamma_k^{\Theta} \times \{ 0\}) \sqcup \left\{
\bigsqcup_{j \in J} (\kappa_j^{\Theta} \times \{ e_j\})
\right\}
\end{equation}
in $\tl{\LL}_{\Theta} \times \RR^{\sharp J}$, where
$e_j=(0, \ldots, 0, 1, 0, \ldots, 0)$ is the $j$-th
standard unit vector in $\RR^{\sharp J}$.
This integral polytope $\Theta_J$ is called the
join of $\kappa_j^{\Theta}$ ($j \in J$) and
$\gamma_k^{\Theta}$. Let $\tl{\Theta}_J$ be the
convex hull of $\{ (0, 0)\} \sqcup \Theta_J$
in $\tl{\LL}_{\Theta} \times \RR^{\sharp J}$.
Then by Definition \ref{AB} we have
\begin{equation}
\beta ( Z^*_{\Theta, J} )_{\lambda}=
\beta ( \Theta_J, \d \Theta + \sharp J )_{\lambda}=
(t^2-1)^{\d \Theta +\sharp J -\d \Theta_J}
\beta ( \Theta_J )_{\lambda}
\end{equation}
for any $\lambda \in \CC^* \setminus \{ 1\}$.
Recall that if $\lambda
\not= 1$ the degree of the polynomial
$\beta ( \Theta_J )_{\lambda} \in \ZZ [t]$ is
$\leq \d \Theta_J$ $=\d (\gamma_k^{\Theta}
+ \sum_{j \in J} \kappa_j^{\Theta})+ \sharp J$.
Then for $\lambda \in \CC^* \setminus \{ 1\}$ we obtain
\begin{eqnarray}
(1-t^2)^{m_{\Theta}}
\beta (Z^*_{\Delta_{\Theta}})_{\lambda}
& = & \frac{(-1)^{m_{\Theta}}}{t^{2k-2}}\sum_{J \subset
\{ 1,2, \ldots, k-1\}}
(t^2-1)^{s_{\Theta}+\sharp J -1-\d \Theta_J}
\beta (\Theta_J )_{\lambda}
\\
& = & \frac{(-1)^{m_{\Theta}}}{t^{2k-2}}\sum_{J \subset
\{ 1,2, \ldots, k-1\}}
\beta (\Theta_J, s_{\Theta}+\sharp J -1 )_{\lambda}.
\end{eqnarray}
Moreover for a non-degenerate hypersurface
$S^*_{\Theta , J}
\subset (\CC^*)^{s_{\Theta}+ \sharp J}$ having the thin
Newton polytope $\tl{\Theta}_J \subset
\RR^{s_{\Theta}+ \sharp J}$ (recall that
$\tl{\LL}_{\Theta} \subset \RR^{s_{\Theta}}$)
and a natural action of
$\mu_{d_{\Theta}}$ on it, we have
\begin{equation}
\beta (\Theta_J, s_{\Theta}+\sharp J -1 )_{\lambda}
=\beta (S^*_{\Theta , J} )_{\lambda}
\end{equation}
for any $\lambda \in \CC^* \setminus \{ 1\}$.
We thus obtain the following theorem.
\begin{theorem}\label{FFC}
For $\lambda \in \CC^* \setminus \{ 1\}$ we have
\begin{eqnarray}
\beta (\SS_{g,0})_{\lambda}
& = &
\dsum_{
\Theta \subset \Gamma_f,
 \d \Theta \geq k-1 }
\frac{(-1)^{m_{\Theta}}}{t^{2k-2}}\sum_{J \subset
\{ 1,2, \ldots, k-1\}}
\beta (\Theta_J, s_{\Theta}+\sharp J -1 )_{\lambda}
\\
& = &
\dsum_{
\Theta \subset \Gamma_f,
 \d \Theta \geq k-1 }
\frac{(-1)^{m_{\Theta}}}{t^{2k-2}}\sum_{J \subset
\{ 1,2, \ldots, k-1\}}
\beta (S^*_{\Theta , J} )_{\lambda}.
\end{eqnarray}
\end{theorem}

\begin{remark}
Although we formulated
Theorem \ref{FFC} for complete
intersection singularities, it can
be readily extended
to the case of isolated
determinantal singularities studied in \cite{E2}
by constructing their toric
resolutions.
See \cite{E2} for the
construction of toric resolutions of
determinantal singularities.
\end{remark}

By Theorems \ref{general}, \ref{FFC}
and the results in Sections
\ref{sec:2}, \ref{sec:3}, we can
calculate the numbers of the
Jordan blocks for the eigenvalues $\lambda \not= 1$
in $\Phi_{n-k,0}: H^{n-k} (F_0;\CC) \simeq
H^{n-k} (F_0;\CC)$ as follows.
First, for a compact face $\Theta \prec
\Gamma_+(f)$ such that $\d \Theta \geq k-1$
and a subset $J \subset
\{ 1,2, \ldots, k-1\}$ we define an
integer $c(\Theta, J) \in \ZZ_+$ by
$c(\Theta, J)= \d \Theta -
\d (\gamma_k^{\Theta}+
\sum_{j \in J} \kappa_j^{\Theta})$.
Then we have $\d \Theta_J=\d \Theta -
c(\Theta, J)+ \sharp J$. Moreover
for $l \geq 1$ we define a finite
subset $R(\Theta_J, l) \subset [0, \d \Theta_J]
\cap \ZZ$ by
\begin{equation}
R(\Theta_J, l)=
\{ 0 \leq r \leq \d \Theta_J \ |\
n+k-3+l \equiv r \ {\rm mod} \  2 \}.
\end{equation}
For each $r \in R(\Theta_J, l)$, we set
\begin{equation}
e(\Theta_J, l)_r :=\frac{n+k-3+l-r}{2}\in \ZZ_+.
\end{equation}
Now in the situation as above, let $\lambda \in
\CC^* \setminus\{1\}$ and $i \geq
1$. Then by Theorems \ref{general} and \ref{FFC},
the number of the Jordan blocks for
the eigenvalue $\lambda$ with
sizes $\geq i$ in $\Phi_{n-k,0}$ is
\begin{equation}
\dsum_{ \begin{subarray}{c}
\Theta \subset \Gamma_f,
 \\
 \d \Theta \geq k-1
 \end{subarray} }
 \dsum_{J \subset \{ 1, \ldots, k-1\}}
(-1)^{n-k+ c(\Theta, J)}
\left\{ \sum_{l=i,i+1} \left(
\dsum_{r \in  R(\Theta_J, l)}
(-1)^{e(\Theta_J, l)_r}
\binom{m_{\Theta}+c(\Theta, J)}{e(\Theta_J, l)_r}
 \beta_r(\Theta_J)_{\lambda}
\right) \right\}.
\end{equation}
Note that we can always calculate
the above virtual Betti numbers
$\beta_r(\Theta_J)_{\lambda} \in \ZZ$
by our algorithm at the end of Section
\ref{sec:3}. We can construct polytopes which
majorize the join $\Theta_J$
much easier than that for arbitrary
polytopes of the same dimension.
From now on,
assume moreover that
for any compact face $\Theta \prec
\Gamma_+(f)$ such that $\d \Theta \geq k-1$
the corresponding faces $\gamma_j^{\Theta}$
($1 \leq j \leq k$) are simplicial and
transversal: $\d (\sum_{j=1}^k \gamma_j^{\Theta})
=\sum_{j=1}^k \d \gamma_j^{\Theta}$.
Note that this condition was used in
\cite{E1} to describe the
difference of the Euler characteristics
of two ``real" Milnor fibers over real
complete intersections. Under this condition,
the join $\Theta_J$ is prime and
hence $\tl{\Theta}_J$ is pseudo-prime.
Therefore, by Proposition \ref{PLP} we have

\begin{eqnarray}
\beta_r(\Theta_J)_{\lambda}
& = &
(-1)^{\d \Theta_J+r}
\sum_{\begin{subarray}{c}
\Gamma \prec \Theta_J, \\ \d \Gamma =r
\end{subarray}} \left\{
\sum_{\gamma \prec \Gamma}
(-1)^{\d \gamma}
\Vol_{\ZZ}(\gamma )_{\lambda}
\right\}
\\
 & = &
(-1)^{\d \Theta_J+r}
\sum_{l=0}^r
\left\{ \sum_{\begin{subarray}{c}
\gamma \prec \Theta_J, \\ \d \gamma =l
\end{subarray}}
(-1)^{l}
\binom{\d \Theta_J -l}{r-l}
\Vol_{\ZZ}(\gamma )_{\lambda}
\right\},
\end{eqnarray}
where by using the lattice distance
$d(\gamma )>0$ of
the face $\gamma \prec \Theta_J$ from the
origin $(0,0) \in
\tl{\LL}_{\Theta} \times \RR^{\sharp J}$
we define the integer
$\Vol_{\ZZ}(\gamma )_{\lambda} \in \ZZ_+$ by
\begin{equation}
\Vol_{\ZZ}(\gamma )_{\lambda}
= \begin{cases}
\Vol_{\ZZ}(\gamma )
& (\lambda^{d(\gamma )}=1), \\
0 & (\text{otherwise}).
\end{cases}
\end{equation}

Finally to end this section, we shall introduce
an analogue of the Steenbrink
conjecture proved by
Varchenko-Khovanskii \cite{K-V} and Saito
\cite{Saito-3}.

\begin{definition}(Ebeling-Steenbrink \cite{E-S}) As
a Puiseux series, we define the
non-integral part $\sp_g(t)$ of the
spectrum of $g: W \longrightarrow \CC$
at the origin $0 \in \CC^n$ by
\begin{equation}
\sp_g(t)
= \sum_{b \in (0,1) \cap \QQ}
\left[ \sum_{i=0}^{n-k}
 \left\{ \sum_{q \geq 0}
e^{i,q}([H_g])_{\exp(-2\pi
\sqrt{-1} b)}\right\}
t^{i+ b} \right] .
\end{equation}
\end{definition}
By Proposition \ref{NJB} (i)
the support of $\sp_g (t)$
is contained in the open interval
$(0,n-k+1)$ and has the symmetry
\begin{equation}
\sp_g(t)=t^{n-k+1} \sp_g ( \frac{1}{t} )
\end{equation}
with center at $\frac{n-k+1}{2}$. Moreover by
the above arguments (the Cayley trick)
and the proof of
\cite[Theorem 5.10]{M-T-4}, we immediately
obtain the following
explicit description of $\sp_g (t)$. For
each $\Theta \prec
\Gamma_{+}(f)$ such that $\Theta \subset
\Gamma_f$ and $J \subset
\{ 1,2, \ldots, k-1\}$ let
$\Cone(\Theta_J)=\RR_+\Theta_J
\subset \tl{\LL}_{\Theta}
\times \RR^{\sharp J}$ be the
cone generated by $\Theta_J$ and
$h_{\Theta, J} \colon
\Cone(\Theta_J) \longrightarrow
\RR$ the linear function such that
$h_{\Theta, J}|_{\Theta_J} \equiv 1$.
Then we define the Puiseux series
$P_{\Theta, J}(t)$ by
\begin{equation}
P_{\Theta, J}(t)=\sum_{b \in \QQ_+
\setminus \ZZ_+}
\sharp \{ v \in
\Cone(\Theta_J) \cap \ZZ^{n+\sharp J}
 \ | \ h_{\Theta, J}(v)
=b \} t^{b}.
\end{equation}

\begin{theorem}\label{SPT}
In the situation as above, we have
\begin{equation}
\sp_g(t)=(-1)^{n-k}  \dsum_{\Theta \subset
\Gamma_f, \d \Theta \geq k-1} \left\{
\dsum_{J \subset \{ 1, \ldots, k-1\}}
(-1)^{\d \Theta + \sharp J}
(1-t)^{s_{\Theta} + \sharp J}
P_{\Theta, J}(t) \right\}
\cdot t^{-k+1}.
\end{equation}
\end{theorem}

\section{The numbers of Jordan blocks
in the monodromies over C.I.}\label{sec:5}

In this section, by using the results in the
previous sections, we prove some combinatorial
formulas for the Jordan normal forms of the
(local) monodromies over complete intersection
subvarieties of $\CC^n$. We inherit the situation
and the notations in Section \ref{sec:4}. Then
our primary interest here is to describe the
numbers of the maximal (and the
second maximal) Jordan blocks for the eigenvalues
$\lambda \not= 1$ in the monodromy
$\Phi_{n-k,0}\colon H^{n-k}(F_0;\CC) \simeq
H^{n-k}(F_0;\CC)$ in terms of the Newton
polyhedrons $\Gamma_+(f_1), \Gamma_+(f_2), \ldots,
\Gamma_+(f_k)$. We fix
$\lambda \in \CC^*\setminus\{ 1\}$ and
 a face $\Theta \prec \Gamma_+(f)$ such
that $\Theta \subset \Gamma_f$. First,
for $J \subset \{ 1,2, \ldots, k-1\}$ let
$K_{\Theta, J}$ be the affine linear subspace
of $\tl{\LL}_{\Theta}
\simeq \RR^{\d \Theta +1}$ which is parallel
to the affine span of the Minkowski sum
$\gamma_k^{\Theta}+
\sum_{j \in J} \kappa_j^{\Theta}$
and contains $\gamma_k^{\Theta}$. Then we
define an integer $d_{\Theta, J}>0$
to be the lattice
distance of $K_{\Theta, J}$ from the
origin $0 \in \tl{\LL}_{\Theta}$. Note
that if $J=\{ 1,2, \ldots, k-1\}$
(resp. $J=\emptyset$) $d_{\Theta, J}$ is
equal to $d_{\Theta}$ (resp. is the
lattice distance $d(\gamma_k^{\Theta})$
of $\gamma_k^{\Theta}$ from
the origin $0 \in \tl{\LL}_{\Theta}$). Moreover,
for any $J \subset \{ 1,2, \ldots, k-1\}$ we see that
$d_{\Theta, J}$ divides $d_{\Theta}$.

\begin{definition}
For $J \subset \{ 1,2, \ldots, k-1\}$
we denote the difference $\d (\gamma_k^{\Theta}+
\sum_{j \in J} \kappa_j^{\Theta})- \sharp J$
by $\delta(\Theta, J)$.
\end{definition}
The following lemma is essentially due to
Sturmfels \cite{Sturmfels}.

\begin{lemma}\label{INE}
If $\delta(\Theta, J) \geq 0$ for any
$J \subset \{ 1,2, \ldots, k-1\}$, then the
set $\{ J \ | \ \delta(\Theta, J)=0 \}$
is closed by unions $\cup$ and intersections
$\cap$. In particular, if moreover
$\{ J \ | \ \delta(\Theta, J)=0 \} \not= \emptyset$
(including the case where $\{ J \ | \ \delta(\Theta, J)=0 \}
= \{ \emptyset \}$), it has the (unique)
maximal element $J_0$.
\end{lemma}
\begin{proof}
Let $I, J \subset \{ 1,2, \ldots, k-1\}$.
Then we can easily prove that
\begin{eqnarray}
 & &
\d (\gamma_k^{\Theta}+\sum_{j \in I}
\kappa_j^{\Theta})
+
\d (\gamma_k^{\Theta}+\sum_{j \in J}
\kappa_j^{\Theta})
\\
& \geq &
\d (\gamma_k^{\Theta}+\sum_{j \in I \cap J}
\kappa_j^{\Theta})
+
\d (\gamma_k^{\Theta}+\sum_{j \in I \cup J}
\kappa_j^{\Theta}).
\end{eqnarray}
Combining this inequality with the one $\sharp I
+\sharp J =\sharp (I \cap J) + \sharp (I \cup J)$
we obtain
\begin{equation}
\delta(\Theta, I)+ \delta(\Theta, J) \geq
\delta(\Theta, I \cap J) +
\delta(\Theta, I \cup J) \geq 0,
\end{equation}
from which the assertion immediately follows.
\end{proof}

\begin{definition}
We define an integer $E(\Theta)_{\lambda}$
to be $0$ if $\min_{J \subset \{ 1,2, \ldots, k-1\}}
\delta(\Theta, J) \not= 0$, and otherwise by
using the maximal element $J_0$ of
$\{ J \ | \ \delta(\Theta, J)=0 \}$ we set
\begin{equation}
E(\Theta)_{\lambda}
=\begin{cases}
MV(\kappa_{j_1}^{\Theta},
\kappa_{j_2}^{\Theta}, \ldots,
\kappa_{j_m}^{\Theta})
 & (\lambda^{d_{\Theta, J_0}}=1),
\\
0 & (\text{otherwise}),
\end{cases}
\end{equation}
where $J_0=\{ j_1, j_2, \ldots, j_m\}$,
$\sharp J_0=m$ and $MV(\kappa_{j_1}^{\Theta},
\ldots, \kappa_{j_m}^{\Theta}) \in \ZZ_+$
is the normalized $m$-dimensional mixed
volume of $\kappa_{j_1}^{\Theta},
\ldots, \kappa_{j_m}^{\Theta}$. Note that
by $\delta(\Theta, J_0)=0$ we have
$\d (\kappa_{j_1}^{\Theta}+
\cdots + \kappa_{j_m}^{\Theta}) \leq \sharp J_0
=m$ and the $m$-dimensional
mixed volume $MV(\kappa_{j_1}^{\Theta},
\ldots, \kappa_{j_m}^{\Theta})$ makes sense.
\end{definition}
In particular, if $\min_{J \subset
\{ 1,2, \ldots, k-1\}}
\delta(\Theta, J) = 0$ and
$\{ J \ | \ \delta(\Theta, J)=0 \}
= \{ \emptyset \}$ we set $E(\Theta)_{\lambda}=1$
or $0$ depending on whether
$\lambda^{d(\gamma_k^{\Theta})}=1$ or not.

\begin{theorem}\label{MPS}
\begin{enumerate}
\item The degree of the virtual Betti polynomial
$\beta (\SS_{g,0})_{\lambda} \in \ZZ [t]$
($\lambda \not= 1$) is bounded by $2n-2k$.
In particular, the sizes of the Jordan blocks
for the eigenvalues $\lambda \not= 1$ in $\Phi_{n-k,0}$
are bounded by $n-k+1$.
\item The number of the Jordan blocks for
the eigenvalue $\lambda \not= 1$
with the maximal possible size
$n-k+1$ in $\Phi_{n-k,0}$ is equal to
\begin{equation}
\sum_{\begin{subarray}{c}
\Theta \subset \Gamma_f, s_{\Theta}=n,
\\ \d \Theta \geq k-1 \end{subarray}}
(-1)^{\d \Theta -(k-1)}E(\Theta)_{\lambda}.
\end{equation}
\end{enumerate}
\end{theorem}
\begin{proof}
\noindent (A) Assume that there exists
$J \subset \{ 1,2, \ldots, k-1\}$ such that
$\delta(\Theta, J)<0$. Then we have
\begin{equation}
\delta(\Theta, J)= \d
(\Delta_{\gamma_k^{\Theta}}+\sum_{j \in J}
\kappa_j^{\Theta})-(\sharp J+1) <0.
\end{equation}
By the dimensional reason, as a non-degenerate C.I.
in $\tl{T}_{\Theta}$ we have
\begin{equation}
\{ g_j^{\Theta}(x) =0 \ (j \in J), \quad
\tl{g}_k^{\Theta}(x)=0 \} =\emptyset.
\end{equation}
Since $Z^*_{\Delta_{\Theta}} \subset
\tl{T}_{\Theta}$ is contained in this set,
we obtain $Z^*_{\Delta_{\Theta}} =
\emptyset$ and hence
\begin{equation}
\frac{(-1)^{m_{\Theta}}}{t^{2k-2}}\sum_{J \subset
\{ 1,2, \ldots, k-1\}}
\beta (\Theta_J, s_{\Theta}
+\sharp J -1 )_{\lambda} = 0.
\end{equation}
\noindent (B) Next assume that $\delta(\Theta, J)
\geq 0$ for any $J \subset \{ 1,2, \ldots, k-1\}$.
Then for any $J \subset \{ 1,2, \ldots, k-1\}$ we have
\begin{eqnarray}
 & &
{\rm deg} \left\{ \frac{(-1)^{m_{\Theta}}}{t^{2k-2}}
\beta (\Theta_J, s_{\Theta}
+\sharp J -1 )_{\lambda} \right\}
\\
& \leq &
2(s_{\Theta}+\sharp J -1-\d \Theta_J) +
\d \Theta_J -2k+2
\\
& = &
2 s_{\Theta}-2k +\sharp J -
\d (\gamma_k^{\Theta}+
\sum_{j \in J} \kappa_j^{\Theta})
\\
& \leq &
2 s_{\Theta}-2k  \leq 2n-2k.
\end{eqnarray}
So the assertion (i) was proved.
By the above calculations, if $s_{\Theta}<n$
or $\delta(\Theta, J)> 0$ for any
$J \subset \{ 1,2, \ldots, k-1\}$ there is no
contribution to the leading coefficient
$\beta_{2n-2k}(\SS_{g,0})_{\lambda}$
from the face $\Theta \prec \Gamma_+(f)$.
Therefore, to prove the assertion (ii),
we have only to consider the compact faces
$\Theta \prec \Gamma_+(f)$ such that
$s_{\Theta}=n$ and $\min_{J \subset
\{ 1,2, \ldots, k-1\}}
\delta(\Theta, J) =0$. In this case, for the
maximal element $J_0$ of the set
$\{ J \ | \ \delta(\Theta, J)=0 \}$ we have
\begin{equation}
J \not\subset J_0 \ \Longrightarrow \
{\rm deg} \left\{ \frac{(-1)^{m_{\Theta}}}{t^{2k-2}}
\beta (\Theta_J, s_{\Theta}
+\sharp J -1 )_{\lambda} \right\} <2n-2k.
\end{equation}
This implies that
\begin{equation}\label{for}
\sum_{J \subset
\{ 1,2, \ldots, k-1\}}
\frac{(-1)^{m_{\Theta}}}{t^{2k-2}}
\beta (\Theta_J, s_{\Theta}
+\sharp J -1 )_{\lambda}
 \equiv
\sum_{J \subset J_0}
\frac{(-1)^{m_{\Theta}}}{t^{2k-2}}
\beta (\Theta_J, s_{\Theta}
+\sharp J -1 )_{\lambda}
\end{equation}
modulo polynomials of degree less than
$2n-2k$. Set $J_0=\{ j_1, j_2, \ldots, j_m\}$,
$\sharp J_0=m$ and let
\begin{equation}
Z^*_{\Delta^{0}_{\Theta}}=
\{ g_{j_1}^{\Theta}(x) = \cdots =g_{j_m}^{\Theta}(x)
=\tl{g}_k^{\Theta}(x)=0 \}
\subset (\CC^*)^{s_{\Theta}}=(\CC^*)^n
\end{equation}
be the non-degenerate complete intersection
subvariety of codimension $m+1$ in
$(\CC^*)^n$ with a natural action of
$\mu_{d_{\Theta}}$. Then by the arguments
in Section \ref{sec:4} we have
\begin{equation}
\sum_{J \subset J_0}
\frac{1}{t^{2m}}
\beta (\Theta_J, n +\sharp J -1 )_{\lambda}
=\beta (Z^*_{\Delta^{0}_{\Theta}} )_{\lambda}.
\end{equation}
Moreover, since by $\delta(\Theta, J_0)=0$
we have $\d (\Delta_{\gamma_k^{\Theta}}
+\sum_{j \in J_0} \kappa_j^{\Theta})=\sharp J_0
+1=m+1$, we can take another $0$-dimensional
non-degenerate complete intersection
subvariety
\begin{equation}
D^*_{\Delta^{0}_{\Theta}}=
\{ g_{j_1}^{\Theta}(x) = \cdots =g_{j_m}^{\Theta}(x)
=\tl{g}_k^{\Theta}(x)=0 \}
\subset (\CC^*)^{m+1}
\end{equation}
in $(\CC^*)^{m+1}$ such that
$Z^*_{\Delta^{0}_{\Theta}}
\simeq (\CC^*)^{n-m-1}
\times D^*_{\Delta^{0}_{\Theta}}$.
This implies that the right hand side of
\eqref{for} is equal to
\begin{equation}
\frac{(-1)^{m_{\Theta}}}{t^{2k-2m-2}}
(t^2-1)^{n-m-1}
\beta (D^*_{\Delta^{0}_{\Theta}}
)_{\lambda}
\end{equation}
whose leading coefficient
(of degree $2n-2k$) is $(-1)^{m_{\Theta}}
\beta (D^*_{\Delta^{0}_{\Theta}}
)_{\lambda}=(-1)^{n-\d \Theta -1}$
$\beta_0 (D^*_{\Delta^{0}_{\Theta}}
)_{\lambda} \in \ZZ$.
By Theorem \ref{BKK}, the subset
\begin{equation}
\{ g_{j_1}^{\Theta}(x) = \cdots =g_{j_m}^{\Theta}(x)
 =0 \} \subset (\CC^*)^{m+1}
\end{equation}
of $(\CC^*)^{m+1}$ is a disjoint union
of $MV(\kappa_{j_1}^{\Theta}, \ldots,
\kappa_{j_m}^{\Theta})$ copies of the complex
line $\CC$ and the restriction of
$\tl{g}_k^{\Theta}$ to each line
vanishes at exactly $d_{\Theta, J_0}$
distinct points.
Moreover we can easily see that
the action of the generator of
$\mu_{d_{\Theta}}$ on these
$d_{\Theta, J_0}$ points corresponds to
that of their automorphism group
$\mu_{d_{\Theta, J_0}}$. Then the
assertion (ii) follows.
This completes the proof.
\end{proof}

Recall that we say polyhedrons
$\Delta_1, \ldots, \Delta_k$
in $\RR^n$ majorize each other if
their dual fans are the same. For example,
if $\Delta_1, \ldots, \Delta_k$ are
similar, they majorize each other.

\begin{corollary}
\begin{enumerate}
\item
If $k=2$, then the number of
the Jordan blocks for
the eigenvalue $\lambda \not= 1$
with the maximal possible size
$n-k+1=n-1$ in $\Phi_{n-k,0}$ is equal to
\begin{equation}
\sum_{\begin{subarray}{c}
\Theta \subset \Gamma_f, s_{\Theta}=n,
\\ \d \Theta =1,
\lambda^{d_{\Theta}}=1 \end{subarray}}
{\rm length}_{\ZZ}(\gamma_1^{\Theta})
+ \sum_{d=2}^{n-1} (-1)^{d-1}
 \sharp \left\{ \Theta \subset \Gamma_f \ | \
\begin{subarray}{c}
 s_{\Theta}=n, \ \ \d \Theta =d,
\\
\d \gamma_2^{\Theta}=0 \
\text{and} \ \lambda^{d(\gamma_2^{\Theta})}=1
 \end{subarray} \right\},
\end{equation}
where ${\rm length}_{\ZZ}(\gamma_1^{\Theta})
\in \ZZ_+$ is the lattice length of the segment
$\gamma_1^{\Theta}$.
\item
If $\Gamma_+(f_1), \ldots, \Gamma_+(f_k)$
majorize each other, then the number of
the Jordan blocks for
the eigenvalue $\lambda \not= 1$
with the maximal possible size
$n-k+1$ in $\Phi_{n-k,0}$ is equal to
\begin{equation}
\sum_{\begin{subarray}{c}
\Theta \subset \Gamma_f, s_{\Theta}=n,
\\ \d \Theta =k-1,
\lambda^{d_{\Theta}}=1 \end{subarray}}
MV(\kappa_{1}^{\Theta}, \ldots,
\kappa_{k-1}^{\Theta}),
\end{equation}
where $MV(\kappa_{1}^{\Theta}, \ldots,
\kappa_{k-1}^{\Theta}) \in \ZZ_+$
is the normalized $(k-1)$-dimensional
mixed volume of $\kappa_{1}^{\Theta}, \ldots,
\kappa_{k-1}^{\Theta}$.
\end{enumerate}
\end{corollary}

From now on, we shall describe the
numbers of the
second maximal Jordan blocks for the eigenvalues
$\lambda \not= 1$ in the monodromy
$\Phi_{n-k,0}\colon H^{n-k}(F_0;\CC) \simeq
H^{n-k}(F_0;\CC)$. We fix
$\lambda \in \CC^*\setminus\{ 1\}$ and
 a face $\Theta \prec \Gamma_+(f)$ such
that $\Theta \subset \Gamma_f$. Recall that
for any $J \subset \{ 1,2, \ldots, k-1\}$
the Minkowski sum
$\gamma_k^{\Theta}+
\sum_{j \in J} \kappa_j^{\Theta}$ majorizes
$\kappa_j^{\Theta}$ ($j \in J$) and
$\gamma_k^{\Theta}$. For a face $\Gamma$ of
$\gamma_k^{\Theta}+
\sum_{j \in J} \kappa_j^{\Theta}$ denote by
$\Gamma_j^{\Theta}$ ($j \in J \sqcup \{ k\}$)
the corresponding faces of
$\kappa_j^{\Theta}$ ($j \in J$) and
$\gamma_k^{\Theta}$. Moreover for such $J$
and $\Gamma$ let
$K_{\Theta, J}^{\Gamma}$ be
the affine linear subspace
of $\tl{\LL}_{\Theta}
\simeq \RR^{\d \Theta +1}$ which is parallel
to the affine span of
$\Gamma$ and contains $\Gamma_k^{\Theta}$. Then we
define an integer $d_{\Theta, J}^{\Gamma}>0$
to be the lattice distance of
$K_{\Theta, J}^{\Gamma}$ from the
origin $0 \in \tl{\LL}_{\Theta}$.

\begin{lemma}\label{MMP}
\begin{enumerate}
\item Assume that
$\min_{J \subset \{ 1,2, \ldots,
k-1\}} \delta (\Theta, J)=1$. Then the
set $\{ J \ | \ \delta(\Theta, J)=1 \}
\not= \emptyset$ has the unique maximal
element.
\item Assume that
$\min_{J \subset \{ 1,2, \ldots,
k-1\}} \delta (\Theta, J)=0$ and let
$J_0$ be the (unique) maximal element
of $\{ J \ | \ \delta(\Theta, J)=0 \}$.
Assume also that the set
$\{ J \ | \ \delta(\Theta, J)=1,
J_0 \subset J \}$ is not empty and let
$I$ and $J$ be its maximal elements.
Then we have $I=J$ or $I \cap J=J_0$.
\end{enumerate}
\end{lemma}
\begin{proof}
\par \noindent (i) Let $I$ and $J$ be
maximal elements of
$\{ J \ | \ \delta(\Theta, J)=1 \}
\not= \emptyset$. Then by the proof of
Lemma \ref{INE} we have
\begin{equation}
2=\delta(\Theta, I)+ \delta(\Theta, J) \geq
\delta(\Theta, I \cap J) +
\delta(\Theta, I \cup J) \geq 2.
\end{equation}
Since $\delta(\Theta, I \cap J),
\delta(\Theta, I \cup J) \geq 1$ we obtain
$\delta(\Theta, I \cup J)=1$. Then by the
maximality of $I$ and $J$ we have $I=J$.
\par \noindent (ii) Assume that $I \not= J$.
Then by the proof of
Lemma \ref{INE} we have
\begin{equation}
2=\delta(\Theta, I)+ \delta(\Theta, J) \geq
\delta(\Theta, I \cap J) +
\delta(\Theta, I \cup J) \geq 0.
\end{equation}
Since $\delta(\Theta, I \cup J) \geq 2$
by the maximality of $I$ and $J$, we
obtain $\delta(\Theta, I \cap J)=0$
and hence $I \cap J=J_0$.
\end{proof}

\begin{definition}
\begin{enumerate}
\item  Assume that
$\min_{J \subset \{ 1,2, \ldots,
k-1\}} \delta (\Theta, J)=1$. Then
we denote by $J_1$ the (unique) maximal
element of $\{ J \ | \ \delta(\Theta, J)=1 \}
\not= \emptyset$.
\item Assume that
$\min_{J \subset \{ 1,2, \ldots,
k-1\}} \delta (\Theta, J)=0$. Then
we define $n_{\Theta} \geq 0$ to be
the number of the maximal elements of
the set $\{ J \ | \ \delta(\Theta, J)=1,
J_0 \subset J \}$. If $n_{\Theta}>0$ we
denote by $J_1, J_2, \ldots, J_{n_{\Theta}}$
the maximal elements.
\end{enumerate}
\end{definition}

\begin{lemma}\label{OGA}
 Assume that
$\min_{J \subset \{ 1,2, \ldots,
k-1\}} \delta (\Theta, J)=0$.
Then for any $J \subset \{ 1,2, \ldots,
k-1\}$ such that $\delta (\Theta, J)=1$
we have $J \subset J_0$ or ($n_{\Theta}>0$
and) $J \subset J_i$ for a unique
$1 \leq i \leq n_{\Theta}$.
\end{lemma}
\begin{proof}
Assume that $\delta (\Theta, J)=1$
and $J \not\subset J_0$. Then by the proof of
Lemma \ref{INE} we have
\begin{equation}
1=\delta(\Theta, J_0)+ \delta(\Theta, J) \geq
\delta(\Theta, J_0 \cap J) +
\delta(\Theta, J_0 \cup J) \geq 0.
\end{equation}
Since $\delta(\Theta, J_0 \cup J) =1$
by the maximality of $J_0$, we
see that $n_{\Theta}>0$
and $J \subset J_0 \cup J \subset
J_i$ for some $1 \leq i \leq n_{\Theta}$.
The uniqueness of $J_i$ such that
$J \subset J_i$ follows from
Lemma \ref{MMP} (ii).
\end{proof}

\begin{definition}
\begin{enumerate}
\item
For $J \subset \{ 1,2, \ldots,
k-1\}$ such that $\delta (\Theta, J)=0$
we set
\begin{equation}
MV(\kappa_{j}^{\Theta}
(j \in J) )_{\lambda}
=\begin{cases}
MV(\kappa_{j_1}^{\Theta}, \ldots,
\kappa_{j_m}^{\Theta})
 & (\lambda^{d_{\Theta, J}}=1),
\\
0 & (\text{otherwise}),
\end{cases}
\end{equation}
where $J=\{ j_1, j_2, \ldots, j_m\}$,
$\sharp J=m$ and $MV(\kappa_{j_1}^{\Theta},
\ldots, \kappa_{j_m}^{\Theta}) \in \ZZ_+$
is the normalized $m$-dimensional mixed volume.
\item  For $J \subset \{ 1,2, \ldots,
k-1\}$ such that $\delta (\Theta, J)=1$
we set
\begin{equation}
MV(\kappa_{j}^{\Theta}
(j \in J), \gamma_k^{\Theta}+
\sum_{j \in J} \kappa_j^{\Theta})_{\lambda}
=\begin{cases}
MV(\kappa_{j_1}^{\Theta}, \ldots,
\kappa_{j_m}^{\Theta},
\gamma_k^{\Theta}+
\sum_{i=1}^m \kappa_{j_i}^{\Theta})
 & (\lambda^{d_{\Theta, J}}=1),
\\
0 & (\text{otherwise}),
\end{cases}
\end{equation}
where $J=\{ j_1, j_2, \ldots, j_m\}$,
$\sharp J=m$ and $MV(\kappa_{j_1}^{\Theta},
\ldots, \kappa_{j_m}^{\Theta},
\gamma_k^{\Theta}+
\sum_{i=1}^m \kappa_{j_i}^{\Theta}) \in \ZZ_+$
is the normalized ($m+1$)-dimensional mixed volume.
\item
For $J \subset \{ 1,2, \ldots,
k-1\}$ such that $\delta (\Theta, J)=1$
and a facet $\Gamma$ of the
($\sharp J +1$)-dimensional
Minkowski sum $\gamma_k^{\Theta}+
\sum_{j \in J} \kappa_j^{\Theta}$
we set
\begin{equation}
MV(\Gamma_{j}^{\Theta}
(j \in J) )_{\lambda}
=\begin{cases}
MV(\Gamma_{j_1}^{\Theta}, \ldots,
\Gamma_{j_m}^{\Theta})
 & (\lambda^{d_{\Theta, J}^{\Gamma}}=1),
\\
0 & (\text{otherwise}),
\end{cases}
\end{equation}
where $J=\{ j_1, j_2, \ldots, j_m\}$,
$\sharp J=m$ and $MV(\Gamma_{j_1}^{\Theta},
\ldots, \Gamma_{j_m}^{\Theta})
 \in \ZZ_+$ is the normalized
$m$-dimensional mixed volume.
\end{enumerate}
\end{definition}

\begin{definition}
For $\lambda \in \CC^*\setminus\{ 1\}$
and a face $\Theta \prec \Gamma_+(f)$ such
that $\Theta \subset \Gamma_f$ we define
an integer $F(\Theta )_{\lambda} \in \ZZ$
as follows.
\begin{enumerate}
\item If $\min_{J \subset \{
1,2, \ldots, k-1\}} \delta (\Theta, J)<0$
or $>1$ we set $F(\Theta )_{\lambda}=0$.
\item If $\min_{J \subset \{
1,2, \ldots, k-1\}} \delta (\Theta, J)=1$,
then by using the maximal element $J_1$ of
$\{ J \ | \ \delta(\Theta, J)=1 \}$ we set
\begin{equation}
F(\Theta )_{\lambda}=
MV(\kappa_{j}^{\Theta}
(j \in J_1), \gamma_k^{\Theta}+
\sum_{j \in J_1} \kappa_j^{\Theta})_{\lambda}
- \sum_{\Gamma} MV(\Gamma_{j}^{\Theta}
(j \in J_1) )_{\lambda},
\end{equation}
where in the sum $\sum_{\Gamma}$ the face
$\Gamma$ ranges through the facets of
$\gamma_k^{\Theta}+
\sum_{j \in J_1} \kappa_j^{\Theta}$.
\item If $\min_{J \subset \{
1,2, \ldots, k-1\}} \delta (\Theta, J)=0$
and $n_{\Theta}=0$ ($\Longleftrightarrow
\{ J \ | \ \delta(\Theta, J)=1, J \not\subset
J_0 \} =\emptyset$ by Lemma \ref{OGA}),
then we set $F(\Theta )_{\lambda}=0$.
\item If $\min_{J \subset \{
1,2, \ldots, k-1\}} \delta (\Theta, J)=0$
and $n_{\Theta}>0$ we set
\begin{eqnarray}\nonumber
 & F(\Theta )_{\lambda}  & =
\\\nonumber
 & \sum_{i=1}^{n_{\Theta}} &
\left\{ 2 MV(\kappa_{j}^{\Theta}
(j \in J_0) )_{\lambda}  +
MV(\kappa_{j}^{\Theta}
(j \in J_i), \gamma_k^{\Theta}+
\sum_{j \in J_i} \kappa_j^{\Theta})_{\lambda}
- \sum_{\Gamma} MV(\Gamma_{j}^{\Theta}
(j \in J_i) )_{\lambda}
\right\},
\end{eqnarray}
where in the sum $\sum_{\Gamma}$ the face
$\Gamma$ ranges through the facets of
$\gamma_k^{\Theta}+
\sum_{j \in J_i} \kappa_j^{\Theta}$.
\end{enumerate}
\end{definition}

\begin{theorem}\label{SMB}
 The number of the Jordan blocks for
the eigenvalue $\lambda \not= 1$
with the second maximal possible size
$n-k$ in $\Phi_{n-k,0}$ is equal to
\begin{equation}
\sum_{\begin{subarray}{c}
\Theta \subset \Gamma_f, s_{\Theta}=n,
\\ \d \Theta \geq k \end{subarray}}
(-1)^{\d \Theta -k}F(\Theta)_{\lambda}.
\end{equation}
\end{theorem}
\begin{proof}
It suffices to calculate the coefficient of
$t^{2n-2k-1}$ of
\begin{equation}
(1-t^2)^{m_{\Theta}}
\beta (Z^*_{\Delta_{\Theta}})_{\lambda}
 =  \frac{(-1)^{m_{\Theta}}}{t^{2k-2}}\sum_{J \subset
\{ 1,2, \ldots, k-1\}}
\beta (\Theta_J, s_{\Theta}+\sharp J -1 )_{\lambda}
\end{equation}
for each face $\Theta \prec \Gamma_+(f)$ such
that $\Theta \subset \Gamma_f$. By the proof of
Theorem \ref{MPS} this coefficient is zero
unless $s_{\Theta}=n$ and
there exists $J \subset \{ 1,2, \ldots,
k-1 \}$ such that $\delta (\Theta, J)
=0, 1$. Here we calculate only the
contribution to the coefficient of
$t^{2n-2k-1}$ from $\Theta$ such
that $\Theta \subset \Gamma_f$,
$s_{\Theta}=n$, $\min_{J \subset \{
1,2, \ldots, k-1\}} \delta (\Theta, J)=0$
and $n_{\Theta}>0$ (other cases can be
treated similarly by using the
proof of Theorem \ref{MPS}). In this case,
by Lemma \ref{OGA} we have
\begin{eqnarray}\nonumber
 & \frac{(-1)^{m_{\Theta}}}{t^{2k-2}} & \sum_{J \subset
\{ 1,2, \ldots, k-1\}}
\beta (\Theta_J, s_{\Theta}+\sharp J -1 )_{\lambda}
\\\nonumber
 & \equiv & \frac{(-1)^{m_{\Theta}}}{t^{2k-2}}
 \left[ \sum_{i=1}^{n_{\Theta}}  \left\{
\sum_{J \subset J_i}
\beta (\Theta_J, n+\sharp J -1 )_{\lambda}
\right\} -
(n_{\Theta}-1)
\sum_{J \subset J_0}
\beta (\Theta_J, n+\sharp J -1 )_{\lambda}
\right]
\end{eqnarray}
modulo polynomials of degree less than
$2n-2k-1$. By the proof of Theorem \ref{MPS} the
contribution to the coefficient of
$t^{2n-2k-1}$ from the second term is zero.
Moreover for each $1 \leq i \leq n_{\Theta}$
the coefficient of $t^{2n-2k-1}$ of
\begin{equation}
\frac{(-1)^{m_{\Theta}}}{t^{2k-2}}
\sum_{J \subset J_i}
\beta (\Theta_J, n+\sharp J -1 )_{\lambda}
\end{equation}
is calculated as follows.
Set $J_i=\{ j_1, j_2, \ldots, j_m\}$,
$\sharp J_i=m$ and let
\begin{equation}
Z^*_{\Delta^{i}_{\Theta}}=
\{ g_{j_1}^{\Theta}(x) = \cdots =g_{j_m}^{\Theta}(x)
=\tl{g}_k^{\Theta}(x)=0 \}
\subset (\CC^*)^{s_{\Theta}}=(\CC^*)^n
\end{equation}
be the non-degenerate complete intersection
subvariety of codimension $m+1$ in
$(\CC^*)^n$ with a natural action of
$\mu_{d_{\Theta}}$. Then by the proof of
Theorem \ref{MPS} we have
\begin{equation}
\beta(Z^*_{\Delta^{i}_{\Theta}})_{\lambda}
= \frac{1}{t^{2m}} \sum_{J \subset J_i}
\beta (\Theta_J, n+\sharp J -1 )_{\lambda}.
\end{equation}
Since by $\delta( \Theta, J_i)
=1$ we have $\d (\Delta_{\gamma_k^{\Theta}}
+\sum_{j \in J_i} \kappa_j^{\Theta})=\sharp J_i
+1+1=m+2$, we can take a
non-degenerate C.I. curve
\begin{equation}
C^*_{\Delta^{i}_{\Theta}}=
\{ g_{j_1}^{\Theta}(x) = \cdots =g_{j_m}^{\Theta}(x)
=\tl{g}_k^{\Theta}(x)=0 \}
\subset (\CC^*)^{m+2}
\end{equation}
in $(\CC^*)^{m+2}$ such that
$Z^*_{\Delta^{i}_{\Theta}}
\simeq (\CC^*)^{n-m-2}
\times C^*_{\Delta^{i}_{\Theta}}$.
Hence we obtain
\begin{equation}
\frac{(-1)^{m_{\Theta}}}{t^{2k-2}}
\sum_{J \subset J_i}
\beta (\Theta_J, n+\sharp J -1 )_{\lambda} =
\frac{(-1)^{n-\d \Theta -1}}{t^{2k-2m-2}}
(t^2-1)^{n-m-2}
\beta(C^*_{\Delta^{i}_{\Theta}})_{\lambda}.
\end{equation}
Since the coefficient of $t^{2n-2k-1}$ of
the last term is $(-1)^{n-\d \Theta -1}
\beta_1(C^*_{\Delta^{i}_{\Theta}})_{\lambda}$,
the assertion follows from the following
proposition.
\end{proof}

\begin{proposition}
In the situation as above, we have
\begin{equation}
\beta_1(C^*_{\Delta^{i}_{\Theta}})_{\lambda}=
-2 MV(\kappa_{j}^{\Theta}
(j \in J_0) )_{\lambda}-
MV(\kappa_{j}^{\Theta}
(j \in J_i), \gamma_k^{\Theta}+
\sum_{j \in J_i} \kappa_j^{\Theta})_{\lambda}
+ \sum_{\Gamma} MV(\Gamma_{j}^{\Theta}
(j \in J_i) )_{\lambda}.
\end{equation}
\end{proposition}
\begin{proof}
By the Cayley trick we have
\begin{equation}
\beta(C^*_{\Delta^{i}_{\Theta}})_{\lambda}=
\frac{1}{t^{2m}}\sum_{J \subset J_i}
\beta( Z_J^* )_{\lambda},
\end{equation}
where $Z_J^*$ is the non-degenerate hypersurface
of $(\CC^*)^{m+2+ \sharp J}$ defined by
\begin{equation}
Z^*_{J}=\{ (x; \alpha_j
(j \in J)) \in (\CC^*)^{m+2} \times
(\CC^*)^{\sharp J}
\ | \ \sum_{j \in J} \alpha_j g_j^{\Theta}(x)
+ \tl{g}_k^{\Theta}(x) = 0 \}
\subset (\CC^*)^{m+2+ \sharp J}.
\end{equation}
By a simple calculation, for any $J \subset J_i$
we see that ${\rm deg} \beta (Z_J^*)_{\lambda}
 \leq 2m+2-\delta (\Theta, J) \leq 2m+2$. Hence
the leading coefficient of the polynomial
$\beta(C^*_{\Delta^{i}_{\Theta}})_{\lambda}
\in \ZZ [t]$
is $\beta_2(C^*_{\Delta^{i}_{\Theta}})_{\lambda}$
and equal to that of
\begin{equation}
\frac{1}{t^{2m}}\dsum_{J \subset J_0}
\beta( Z_J^* )_{\lambda}=
\frac{1}{t^{2(m- \sharp J_0)}}
(t^2-1)^{m-\sharp J_0+1}
\beta(D^*_{\Delta^{0}_{\Theta}})_{\lambda},
\end{equation}
where $D^*_{\Delta^{0}_{\Theta}} \subset
(\CC^*)^{\sharp J_0 +1}$ is the $0$-dimensional
non-degenerate C.I. in $(\CC^*)^{\sharp J_0 +1}$
in the proof of Theorem \ref{MPS}.
Consequently we obtain
\begin{equation}
\beta_2(C^*_{\Delta^{i}_{\Theta}})_{\lambda}=
\beta_0(D^*_{\Delta^{0}_{\Theta}})_{\lambda}=
MV( \kappa_j^{\Theta} (j \in J_0) )_{\lambda}.
\end{equation}
From now on, let us calculate the
$\lambda$-Euler characteristic
\begin{equation}
\chi (C^*_{\Delta^{i}_{\Theta}})_{\lambda}=
\beta_0(C^*_{\Delta^{i}_{\Theta}})_{\lambda}+
\beta_1(C^*_{\Delta^{i}_{\Theta}})_{\lambda}+
\beta_2(C^*_{\Delta^{i}_{\Theta}})_{\lambda}
\end{equation}
of the C.I. curve
$C^*_{\Delta^{i}_{\Theta}}
\subset (\CC^*)^{m+2}$. First, note that
if $J$ satisfies the condition
$\d (\Delta_{\gamma_k^{\Theta}}
+\sum_{j \in J} \kappa_j^{\Theta})<m+2$
we have $\chi (Z_J^*)_{\lambda}=
\chi (Z_J^*)=0$. So in the sum
\begin{equation}
\chi (C^*_{\Delta^{i}_{\Theta}})_{\lambda}=
\sum_{J \subset J_i} \chi (Z_J^*)_{\lambda}
\end{equation}
only the terms $\chi (Z_J^*)_{\lambda}$
for $J$ such that $\d ( \gamma_k^{\Theta}
+\sum_{j \in J} \kappa_j^{\Theta})=m+1$
can be non-trivial. By Theorem \ref{LP}
this implies that we have
\begin{equation}
\chi (C^*_{\Delta^{i}_{\Theta}})_{\lambda}
=\begin{cases}
\frac{1}{d_{\Theta, J_i}}
\chi (C^*_{\Delta^{i}_{\Theta}})
 & ( \lambda^{d_{\Theta, J_i}}=1  ), \\
0 & ( \text{otherwise}  ).
\end{cases}
\end{equation}
Since $C^*_{\Delta^{i}_{\Theta}}
\subset (\CC^*)^{m+2}$ is a
$d_{\Theta, J_i}$-fold covering of
the C.I. curve
\begin{equation}
\{ g_{j_1}^{\Theta}(x) = \cdots =g_{j_m}^{\Theta}(x)
=0 \} \setminus \{ g_{j_1}^{\Theta}(x) =
 \cdots =g_{j_m}^{\Theta}(x)
= g_k^{\Theta}(x)=0 \}
\end{equation}
in $(\CC^*)^{m+1}$, its usual Euler
characteristic
$\chi (C^*_{\Delta^{i}_{\Theta}})$
is calculated by Theorem \ref{BKK} as
\begin{equation}
\chi (C^*_{\Delta^{i}_{\Theta}})
= - d_{\Theta, J_i} \times
MV(\kappa_{j_1}^{\Theta}, \ldots,
\kappa_{j_m}^{\Theta},
\gamma_k^{\Theta}+
\sum_{i=1}^m \kappa_{j_i}^{\Theta}).
\end{equation}
Hence we get
\begin{equation}
\chi (C^*_{\Delta^{i}_{\Theta}})_{\lambda}=
 - MV(\kappa_{j}^{\Theta}
(j \in J_i), \gamma_k^{\Theta}+
\sum_{j \in J_i} \kappa_j^{\Theta})_{\lambda}.
\end{equation}
Now denote by $\Box$ the Minkowski sum
$\Delta_{\gamma_k^{\Theta}} +
\sum_{j \in J_i} \kappa_j^{\Theta}$ in $\RR^{m+2}$. Then
$\Box$ is an ($m+2$)-dimensional polytope
and majorizes $\kappa_j^{\Theta}$ ($j \in J_i$)
and $\Delta_{\gamma_k^{\Theta}}$. Let $X_{\Box}$
be the toric variety associated to
the dual fan of $\Box$. Then
$X_{\Box}$ is a compactification of the
complex torus
$T=(\CC^*)^{m+2}$ and smooth outside the union
of $T$-orbits of codimension
$\geq 2$. Therefore, by the non-degeneracy of
$C^*_{\Delta^{i}_{\Theta}} \subset (\CC^*)^{m+2}$
its closure $\overline{ C^*_{\Delta^{i}_{\Theta}} }$
in $X_{\Box}$ is a smooth projective curve.
By the Poincar{\'e} duality of
$\overline{ C^*_{\Delta^{i}_{\Theta}} }$ we have
\begin{equation}
\beta_0(\overline{
C^*_{\Delta^{i}_{\Theta}}})_{\lambda}=
\beta_2(\overline{
C^*_{\Delta^{i}_{\Theta}}})_{\lambda^{-1}}=
\beta_2(
C^*_{\Delta^{i}_{\Theta}})_{\lambda^{-1}}=
MV(\kappa_{j}^{\Theta}
(j \in J_0) )_{\lambda}
\end{equation}
and hence
\begin{equation}
\beta_0(C^*_{\Delta^{i}_{\Theta}})_{\lambda}=
MV(\kappa_{j}^{\Theta}
(j \in J_0) )_{\lambda}-
\sum_{\Gamma} MV(\Gamma_{j}^{\Theta}
(j \in J_i) )_{\lambda}.
\end{equation}
Then we obtain the desired formula for the
first virtual Betti number
$\beta_1(C^*_{\Delta^{i}_{\Theta}})_{\lambda}$.
This completes the proof.
\end{proof}

\begin{corollary}
\begin{enumerate}
\item If $k=2$, then the number of
the Jordan blocks for
the eigenvalue $\lambda \not= 1$
with the second maximal possible size
$n-k$ in $\Phi_{n-k,0}$ is equal to
\begin{equation}
\sum_{\begin{subarray}{c}
\Theta \subset \Gamma_f, s_{\Theta}=n,
\\ \d \Theta =2 \end{subarray}}
F(\Theta )_{\lambda} +\sum_{d=3}^{n-1} (-1)^d
\left\{
\sum_{\begin{subarray}{c}
s_{\Theta}=n, \d \Theta =d,
\\ \d \gamma_2^{\Theta} =1,
\lambda^{d(\gamma_2^{\Theta})}=1
\end{subarray}}
\left(
{\rm length}_{\ZZ}(\gamma_2^{\Theta})
-  \sharp \{ v \prec \gamma_2^{\Theta} \ | \
\d v =0, \lambda^{d(v)}=1 \}
\right) \right\} .
\end{equation}
\item
If $\Gamma_+(f_1), \ldots, \Gamma_+(f_k)$
majorize each other, then the number of
the Jordan blocks for
the eigenvalue $\lambda \not= 1$
with the second maximal possible size
$n-k$ in $\Phi_{n-k,0}$ is equal to
\begin{equation}
\sum_{\begin{subarray}{c}
\Theta \subset \Gamma_f, s_{\Theta}=n,
\\ \d \Theta =k \end{subarray}}
F(\Theta )_{\lambda}.
\end{equation}
\end{enumerate}
\end{corollary}

\section{Monodromies at infinity over C.I.}\label{sec:6}

In this section, we study the monodromies at
infinity over complete intersection subvarieties
in $\CC^n$. For $2 \leq k \leq n$ let
\begin{equation}
W= \{ f_1= \cdots =f_{k-1}=0 \} \supset
V= \{ f_1= \cdots =f_{k-1}=f_k=0 \}
\end{equation}
be complete intersection subvarieties
of $\CC^n$. Then for the polynomial map
$f=(f_1,f_2, \ldots, f_k) : \CC^n \longrightarrow \CC^k$,
there exists a complex hypersurface
$D \subset \CC^k$ such that the restriction
$\CC^n \setminus f^{-1}(D)
\longrightarrow \CC^k \setminus D$ of $f$ is
a locally trivial
fibration. We assume that the $k$-th coordinate axis
\begin{equation}
A_k=\{ y \in \CC^k \ | \ y_1=y_2=\cdots =y_{k-1}=0\}
\simeq \CC
\end{equation}
satisfies the condition $\sharp (A_k \cap D)<+ \infty$.
Let $C_R$ be a circle in $A_k \simeq
\CC$ centered at the origin $0 \in A_k$
with a sufficiently large radius $R>>0$. Here we
take $R$ large enough so that the open
disk whose boundary is $C_R$
contains the finite set $A_k \cap D$.
Let $g=f|_W=f_k|_W: W=f^{-1}(A_k)
\longrightarrow \CC \simeq
A_k$ be the restriction of $f$ to $W$.
Then by restricting the locally
trivial fibration $W \setminus g^{-1}(A_k \cap D)
\longrightarrow A_k \setminus
(A_k \cap D)$ to $C_R \subset \CC \simeq A_k$ we
obtain a geometric monodromy
automorphism $\Phi^{\infty} :
 g^{-1}(R) \simto g^{-1}(R)$ and the
linear maps
\begin{equation}
\Phi_j^{\infty} : H^j(g^{-1}(R) ;\CC)
\overset{\sim}{\longrightarrow}
H^j(g^{-1}(R) ;\CC) \ \qquad \ (j=0,1,\ldots )
\end{equation}
associated to it. We call $\Phi_j^{\infty}$
the (cohomological) $k$-th principal
monodromies at infinity of $f$.
The semisimple parts of $\Phi_j^{\infty}$
were studied in \cite[Section 5]{M-T-3}.
From now on, we shall determine their
Jordan normal forms for the eigenvalues
$\lambda \not= 1$. Note that if the generic fiber
$g^{-1}(R)$ ($R>>0$) of $g$ satisfies
the condition $H^j(g^{-1}(R) ;\CC) \simeq 0$
($j\neq 0, \ n-k$) (see e.g.
Tib{\u a}r \cite[Theorem 6.2]{Tibar})
then it suffices
to determine the Jordan normal form of
$\Phi_{n-k}^{\infty} : H^{n-k}(g^{-1}(R) ;\CC)
\simto H^{n-k}(g^{-1}(R) ;\CC)$ ($R>>0$).
Let $j: \CC \hookrightarrow \PP^1$ be the
inclusion and $h$ a local coordinate on a
neighborhood of $\infty \in \PP^1$ such that
$\{ \infty \} =h^{-1}(0)$. Then to the object
$\psi_{h}(j_! Rg_!(\CC_{W}))
\in \Dbc(\{ \infty \})$
and the semisimple part of the
monodromy automorphism acting on it, we
associate naturally an element
\begin{equation}
[H_g^{\infty}] \in \KK_0(\HSm).
\end{equation}
\begin{proposition}\label{NJBI}
Assume that $W \subset \CC^n$ has only isolated
singular points, $g: W \longrightarrow \CC$ is
cohomologically tame in the
sense of Sabbah \cite{Sabbah-2}
and the generic fiber
$g^{-1}(R)$ ($R>>0$) of $g$ satisfies
the condition $H^j(g^{-1}(R) ;\CC) \simeq 0$
($j\neq 0, \ n-k$).
Let $\lambda \in \CC^* \setminus \{1\}$. Then
\begin{enumerate}
\item We have $e^{p,q}(
[H_g^{\infty}])_{\lambda}=0$ for $(p,q)
\notin [0,n-k] \times [0,n-k]$.
Moreover for $(p,q) \in [0,n-k] \times [0,n-k]$ we have
\begin{equation}
e^{p,q}( [H_g^{\infty}])_{\lambda}=e^{n-k-q,n-k-p}(
[H_g^{\infty}])_{\lambda}.
\end{equation}
\item For $i \geq 1$, the number
of the Jordan blocks
for the eigenvalue $\lambda$ with sizes $\geq i$
in $\Phi_{n-k}^{\infty} : H^{n-k}(g^{-1}(R) ;\CC)
\simto H^{n-k}(g^{-1}(R) ;\CC)$ ($R>>0$) is equal to
\begin{equation}
(-1)^{n-k} \sum_{p+q=n-k-1+i, n-k+i} e^{p,q}(
[H_g^{\infty}])_{\lambda}.
\end{equation}
\end{enumerate}
\end{proposition}
\begin{proof}
Let $\F \in \Dbc(W)$ be the intersection cohomology
complex of $W$. Then by using a nice compactification
of $g: W \longrightarrow \CC$ we can prove an
analogue of \cite[Theorem 8.1 (ii)]{Sabbah-2} for
the natural morphism $Rg_! \F \longrightarrow
Rg_* \F$. Hence for $\lambda \not= 1$ the $\lambda$-part
$\psi_{h, \lambda}(j_! Rg_!(\CC_{W}))
\in \Dbc(\{ \infty \})$ is isomorphic to
\begin{equation}
\psi_{h, \lambda}(j_! Rg_!(\F [-n+k-1])) \simeq
\psi_{h, \lambda}(Rj_* Rg_*(\F [-n+k-1])).
\end{equation}
Then by the proof of Sabbah \cite[Theorem 13.1]{Sabbah-2}
its relative monodromy
filtration is the absolute one (up to a shift),
and the assertions follow.
\end{proof}
In order to rewrite this result
explicitly, assume moreover that
the polynomials $f_1, f_2, \ldots, f_k$ are
convenient.

\begin{definition}[\cite{L-S},
\cite{M-T-3} etc.]\label{dfn:3-1}
\begin{enumerate}
\item For $1 \leq j \leq k$ we
call the convex hull of $\{0\} \cup NP(f_j)$
in $\RR^n$ the Newton
polyhedron at infinity of $f_j$ and
denote it by $\Gamma_{\infty}(f_j)$.
Moreover we set
\begin{equation}
\Gamma_{\infty}(f):=\Gamma_{\infty}(f_1)
+\Gamma_{\infty}(f_2)+
\cdots + \Gamma_{\infty}(f_k).
\end{equation}
\item We say that $\Theta \prec
\Gamma_{\infty}(f)$ (resp. $\gamma \prec
\Gamma_{\infty}(f_j)$) is a face at infinity
if $0 \notin \Theta$ (resp. $0 \notin \gamma$).
\end{enumerate}
\end{definition}
As in Section \ref{sec:4}, for each face at infinity
$\Theta$ of $\Gamma_{\infty}(f)$ we define those
$\gamma_j^{\Theta}$ of $\Gamma_{\infty}(f_j)$ so that
we have
\begin{equation}
\Theta = \gamma_1^{\Theta} + \gamma_2^{\Theta}
+ \cdots + \gamma_k^{\Theta}.
\end{equation}
For $1 \leq j \leq k$ and a face at infinity
$\Theta$ of $\Gamma_{\infty}(f)$
we define $f_j^{\Theta}
\in \CC [x_1, x_2, \ldots, x_n]$ as in
Section \ref{sec:4}.

\begin{definition}\label{dfn:3-3} (\cite{M-T-3})
We say that $f=(f_1, \ldots, f_k)$ is
non-degenerate at infinity if for
any face at infinity $\Theta$ of $\Gamma_{\infty}(f)$
the two subvarieties
$\{ f_1^{\Theta}(x)= \cdots =f_{k-1}^{\Theta}(x)=0 \}$
and
$\{ f_1^{\Theta}(x)= \cdots =f_{k-1}^{\Theta}(x)=
f_{k}^{\Theta}(x)=0 \}$
in $(\CC^*)^n$
are non-degenerate complete intersections.
\end{definition}
From now on, let us assume also that $f$ is
non-degenerate at infinity. Let $\Sigma_1$
be the dual fan of $\Gamma_{\infty}(f)$ and
$\Sigma_0$ the fan formed by the faces of
the first quadrant $\RR_+^n$. By the
convenience of $f_1, \ldots, f_k$, $\Sigma_0$
is a subfan of $\Sigma_1$ and hence we can
construct a smooth subdivision $\Sigma$ of
$\Sigma_1$ without subdividing the cones in
$\Sigma_0$. Then the toric variety $X_{\Sigma}$
associated to $\Sigma$ is a smooth compactification
of $\CC^n$. By the non-degeneracy of $f$ at infinity,
it follows from the construction of
$X_{\Sigma}$ that $W \subset \CC^n$
has only isolated singular points. By
constructing a blow-up of $X_{\Sigma}$
(to eliminate the points of indeterminacy
of the meromorphic extension of $f_k$ to
$X_{\Sigma}$) as in
\cite{M-T-4}, we can check also that
$g: W \longrightarrow \CC$ is
cohomologically tame. Moreover by
\cite[Theorem 6.2]{Tibar} the generic fiber
$g^{-1}(R)$ ($R>>0$) of $g$ satisfies
 $H^j(g^{-1}(R) ;\CC) \simeq 0$
($j\neq 0, \ n-k$). Hence all the assumptions
of Proposition \ref{NJBI} are satisfied.
Now, as in \cite[Section 4]{M-T-4}
and \cite{Raibaut} (see also
\cite{G-L-M}), by using the
blow-up of $X_{\Sigma}$ we can
construct an element $\SS_g^{\infty}
\in \M_{\CC}^{\hat{\mu}}$
such that
$\chi_h(\SS_g^{\infty})=[H_g^{\infty}]$.
We call $\SS_g^{\infty}$ the
motivic Milnor fiber at infinity of $g : W
\longrightarrow \CC$. For each face at infinity
$\Theta$ of $\Gamma_{\infty}(f)$ we define
an element $[Z^*_{\Delta_{\Theta}}] \in
\M_{\CC}^{\hat{\mu}}$ and an integer
$m_{\Theta} \in \ZZ_+$ etc. as in
Section \ref{sec:4}. Then the following result
can be obtained in the same way as
Theorem \ref{general} (see the proof of
\cite[Theorems 5.3 and 7.3]{M-T-4}).

\begin{theorem}\label{GTI}
Assume that $\lambda \in \CC^*\setminus \{1\}$. Then
\begin{enumerate}
\item In the
Grothendieck group $\KK_0(\HS )$, we have
\begin{equation}
\chi_h(\SS_g^{\infty})_{\lambda}= \dsum_{
 0 \notin \Theta,
 \d \Theta \geq k-1 }
\chi_h\big((1-\LL)^{m_{\Theta}}
\cdot[Z_{\Delta_{\Theta}}^*]\big)_{\lambda}.
\end{equation}
In particular, we have
\begin{equation}
\beta (\SS_g^{\infty})_{\lambda}= \dsum_{
 0 \notin \Theta,
 \d \Theta \geq k-1 }
(1-t^2)^{m_{\Theta}} \cdot
\beta (Z^*_{\Delta_{\Theta}})_{\lambda}.
\end{equation}
\item For $i \geq 1$, the number
of the Jordan blocks for the eigenvalue
$\lambda$ with sizes $\geq i$ in
$\Phi_{n-k}^{\infty} : H^{n-k}(g^{-1}(R) ;\CC)
\simto H^{n-k}(g^{-1}(R) ;\CC)$ ($R>>0$)
is equal to
\begin{equation}
(-1)^{n-k} \left\{ \beta_{n-k-1+i}(\SS_{g}^{\infty})_{\lambda}
+ \beta_{n-k+i}(\SS_{g}^{\infty})_{\lambda} \right\} .
\end{equation}
\end{enumerate}
\end{theorem}
Moreover, also for $\beta (\SS_g^{\infty})_{\lambda}
\in \ZZ [t]$ ($\lambda \not= 1$)
we can obtain a formula completely
similar to Theorem \ref{FFC}. Therefore, we can
always calculate the numbers of the
Jordan blocks for the eigenvalues $\lambda \not= 1$
in $\Phi_{n-k}^{\infty}$ by the
results in Sections \ref{sec:2}
and \ref{sec:3}.
It is also clear that
the analogues of the results in Sections \ref{sec:4}
and \ref{sec:5} hold for
$\Phi_{n-k}^{\infty}$. We thus find
a striking symmetry between local and global.

\end{document}